\documentclass[UTF-8,reqno]{amsart}
\usepackage{enumerate}
\usepackage{amssymb,url,color, booktabs}
\usepackage{mathrsfs}

\usepackage{amsmath} 
\usepackage{amsfonts}
\usepackage{amsthm}

\usepackage{mathtools}                                                      
\mathtoolsset{showonlyrefs=true}

\usepackage{color}
\usepackage[colorlinks=true]{hyperref}
\hypersetup{
    linkcolor=blue,          
    citecolor=red,        
    filecolor=blue,      
    urlcolor=cyan
}

\numberwithin{equation}{section}

\newcommand{\be}{\begin{eqnarray}}
\newcommand{\ee}{\end{eqnarray}}
\newcommand{\ce}{\begin{eqnarray*}}
\newcommand{\de}{\end{eqnarray*}}
\newtheorem{theorem}{Theorem}[section]
\newtheorem{lemma}[theorem]{Lemma}
\newtheorem{remark}[theorem]{Remark}
\newtheorem{definition}[theorem]{Definition}
\newtheorem{proposition}[theorem]{Proposition}
\newtheorem{Examples}[theorem]{Example}
\newtheorem{corollary}[theorem]{Corollary}

\def\wt{\widetilde}

\def\eps{\varepsilon}

\def\e{\mathrm{e}}

\def\p{\partial}

\def\[{{\Big[}}
\def\]{{\Big]}}
\def\<{{\langle}}
\def\>{{\rangle}}
\def\({{\Big(}}
\def\){{\Big)}}

\def\bx{{\mathbf{x}}}

\def\dif{{\mathord{{\rm d}}}}

\def\no{\nonumber}
\def\={&\!\!=\!\!&}

\def\cB{{\mathcal B}}
\def\cC{{\mathcal C}}

\def\cL{{\mathcal L}}

\def\cP{{\mathcal P}}

\def\mD{{\mathbb D}}
\def\mE{{\mathbb E}}

\def\mI{{\mathbb I}}

\def\mN{{\mathbb N}}

\def\mP{{\mathbb P}}
\def\mQ{{\mathbb Q}}
\def\mR{{\mathbb R}}

\def\1{{\mathbf{1}}}

\def\sB{{\mathscr B}}

\def\sF{{\mathscr F}}

\def\sL{{\mathscr L}}
\def\sM{{\mathscr M}}

\def\geq{\geqslant}
\def\leq{\leqslant}

\def\div{\mathord{{\rm div}}}

\def\eps{\varepsilon}

\def\e{\mathrm{e}}

\def\p{\partial}

\def\[{{\Big[}}
\def\]{{\Big]}}
\def\<{{\langle}}
\def\>{{\rangle}}
\def\({{\Big(}}
\def\){{\Big)}}

\def\bx{{\mathbf{x}}}

\def\dif{{\mathord{{\rm d}}}}

\def\no{\nonumber}
\def\={&\!\!=\!\!&}
\def\bt{\begin{theorem}}
\def\et{\end{theorem}}
\def\bl{\begin{lemma}}
\def\el{\end{lemma}}
\def\br{\begin{remark}}
\def\er{\end{remark}}
\def\bx{\begin{Examples}}
\def\ex{\end{Examples}}
\def\bd{\begin{definition}}
\def\ed{\end{definition}}
\def\bp{\begin{proposition}}
\def\ep{\end{proposition}}
\def\bc{\begin{corollary}}
\def\ec{\end{corollary}}

\def\geq{\geqslant}
\def\leq{\leqslant}

\def\KK{q}
\def\div{\mathord{{\rm div}}}

\def\ff{\frac}   \def\ff{\frac}

\def\<{\langle} \def\>{\rangle}

\allowdisplaybreaks

\begin{document}

\title[Ergodicity of supercritical  SDEs  driven by $\alpha$-stable processes]
{Ergodicity of supercritical  SDEs  driven by $\alpha$-stable processes and heavy-tailed sampling}
\date{}

\author{Xiaolong Zhang,\ \ Xicheng Zhang}

\address{Xiaolong Zhang:
School of Mathematics and Statistics, Wuhan University,
Wuhan, Hubei 430072, P.R.China\\
Email: zhangxl.math@whu.edu.cn}

\address{Xicheng Zhang:
School of Mathematics and Statistics, Wuhan University,
Wuhan, Hubei 430072, P.R.China\\
Email: XichengZhang@gmail.com
}

\thanks{
This work is partially supported by NNSFC grants of China (Nos. 12131019, 11731009), and the German Research Foundation (DFG) through the Collaborative Research 
Centre(CRC) 1283 ``Taming uncertainty and profiting from randomness and low regularity in analysis, stochastics and their applications".
}

\begin{abstract}
Let $\alpha\in(0,2)$ and $d\in\mathbb{N}$. Consider the following stochastic differential equation (SDE) driven by $\alpha$-stable process in $\mathbb{R}^d$:
$$
\dif X_t=b(X_t)\dif t+\sigma(X_{t-})\dif L^{\alpha}_t, \quad X_0=x\in\mathbb{R}^d,
$$
where $b:\mathbb{R}^d\to\mathbb{R}^d$ and $\sigma:\mathbb{R}^d\to\mathbb{R}^d\otimes\mathbb{R}^d$ are locally $\gamma$-H\"older continuous  with $\gamma\in(0\vee(1-\alpha)^+,1]$, 
$L^\alpha_t$ is a $d$-dimensional rotationally invariant $\alpha$-stable process. Under some dissipative and non-degenerate assumptions on $b,\sigma$, 
we show the $V$-uniformly exponential ergodicity for the semigroup $P_t$ associated with $\{X_t(x),t\geq 0\}$. Our proofs are mainly based on the heat kernel estimates recently established in \cite{MZ20}
through showing the strong Feller property and the irreducibility of $P_t$. It is interesting that when $\alpha$ goes to zero, the diffusion coefficient 
$\sigma$ can grow faster than drift $b$. 
As applications, we put forward a new heavy-tailed sampling scheme.


\end{abstract}

\maketitle \rm


\section{Introduction}

Let $\pi(x)\varpropto \e^{-U(x)}$ be a probability density function, where $U:\mR^d\to\mR$ is a potential function. 
The classical Langevin Monte Carlo (LMC) method is to use the Langevin diffusion processes to approximate the target distribution $\mu(\dif x)=\pi(x)\dif x$.
More precisely, consider the following stochastic differential equations (SDEs):
\begin{align}\label{SDE6}
\dif X_t=-\nabla U(X_t)\dif t+\sqrt{2}\dif B_t,\ \ X_0=x,
\end{align}
where $(B_t)_{t\geq 0}$ is a $d$-dimensional standard Brownian motion. 
Under some regularity assumptions on $U$, it is well known that the above SDE admits a unique solution $(X_t)_{t\geq 0}$ and 
$\mu$ is an invariant measure of the semigroup $P_t$ generated by $X_t(x)$.
Moreover, under the following dissipativity assumption: for some $c_0,c_1>0$,
$$
\<x, \nabla U(x)\>\geq c_0|x|^2-c_1,
$$
it is also well known that the law of $X_t$ exponentially converges to the stationary distribution $\mu$ in some sense 
as $t\to\infty$ (cf. \cite{BEL18, MT09}), i.e., $P_t$ is exponentially ergodic. 
In particular, it provides a way to generate samples from $\mu$ by the Euler discretization for SDE \eqref{SDE6} (cf. \cite{BEL18}). 
Nowadays, there are numerous works about the analysis of the performance of LMC algorithms (see \cite{DK19, DT12, GC11} and references therein).

\medskip

On the other hand, in statistics physics, 
the Langevin diffusion $X_t$  represents the position of a particle at time $t$ which is influenced by some random force, where the random force is usually the sum of many i.i.d. random pulses with {\it finite} variance. 
By the central limit theorem, the sum of the pulses
converges to a Gaussian distribution. While, if the random pulses have {\it infinite} variance, then the sum of the pulses would converge to an $\alpha$-stable distribution, a typical class of heavy-tailed distribution.
Moreover, it was already shown in \cite{Rob96} that for $d=1$ and $U(x)=|x|^\beta$, the diffusion process defined by SDE \eqref{SDE6} is exponentially ergodic if and only if $\beta\geq 1$.
Therefore, using continuous Langevin diffusion \eqref{SDE6} to simulate the heavy-tailed distribution is not any more a good choice. 
Thus, instead of SDE \eqref{SDE6}, it is quite natural to consider the following SDE:
\begin{align}\label{SDE7}
\dif X_t=b(X_t)\dif t+\dif L^\alpha_t,\ \ \alpha\in(0,2),
\end{align}
where $(L^\alpha_t)_{t\geq0}$ is a $d$-dimensional rotationally invariant $\alpha$-stable process with generator $\Delta^{\frac{\alpha}{2}}$, the usual fractional Laplacian, and
$$
b(x)=-\frac{\Gamma(\frac{d-2+\alpha}{2})}{\pi^{\frac{d}{2}}2^{\frac{2-\alpha}{2}}\Gamma(\frac{2-\alpha}{2})}\e^{U(x)}\int_{\mR^d}\frac{\e^{-U(y)}\nabla U(y)}{|x-y|^{d-(2-\alpha)}}\dif y=\e^{U(x)}\Delta^{\frac{\alpha-2}{2}}(\nabla\e^{-U})(x),
$$
where $d\geq 2$ and $\Delta^{\frac{\alpha-2}{2}}$ is the Riesz potential (cf. \cite[page117]{S79}). Formally, one sees that
$$
\Delta^{\frac{\alpha}{2}}\e^{-U}-\div(b\e^{-U})\equiv 0,
$$
which implies that $\mu$ is an invariant measure of SDE \eqref{SDE7}.
Rigorous theoretical results for the exponential ergodicity of the above SDE \eqref{SDE7} were given in \cite[Theorem 1.1]{HMW21}, and therein,
some explicit conditions on $U$ are provided.
This approach called Fractional Langevin Monte Carlo was firstly introduced in \cite{S17, NSR19}. It has been demonstrated to be useful in modern machine learning both in optimization and heavy-tailed sampling \cite{GSZ21, SZTG20, YZ18}.

\medskip

Motivated by the above mentioned works, in this paper we consider the following SDE driven by multiplicative $\alpha$-stable noises:
\begin{align}
\label{sde:main}
\dif X_t=b(X_t)\dif t+\sigma(X_{t-})\dif L^{\alpha}_t,\ \  X_0=x,
\end{align}
where  $b:\mathbb{R}^d\to\mathbb{R}^d$ and $\sigma:\mathbb{R}^d\to\mathbb{R}^d\otimes\mathbb{R}^d$ are two Borel measurable functions. 
The main contributions of this paper are two-folds:

\begin{enumerate}[(i)]
\item Under some dissipativity, non-degeneracy and locally H\"older regularity assumptions on $b$ and $\sigma$, we establish 
the exponential ergodicity of SDE \eqref{sde:main} for all $\alpha\in(0,2)$, which, in particular, covers the supercritical regime $\alpha\in(0,1)$. 

\item For a large class of heavy-tailed distributions, we put forward a new ergodicity SDE driven by multiplicative $\alpha$-stable 
processes to simulate it. Compared with the recent work \cite{HMW21}, our conditions on the potential functions are simpler and easier to realize in the computer.
\end{enumerate}

Now we make the following assumptions on $b$ and $\sigma$:
\begin{enumerate}[{\bf (H$_{\rm loc}$)}]
\item For any $m\in\mN$, there is a constant $C_m\geq 1$ such that for all $|x|\leq m$ and $\xi\in\mR^d$,
$$
C_m^{-1}|\xi|\leq |\sigma(x)\xi|\leq C_m|\xi|.
$$
Moreover, there are $\gamma\in((1-\alpha)^+,1]$ and locally bounded measurable functions $\ell_1$ and $\ell_2$ on $\mR^d$ such that for all $x\in\mR^d$ and $|h|\leq 1$,
\begin{align}\label{Hol}
|b(x+h)-b(x)|\leq \ell_1(x)|h|^\gamma,\ |\sigma(x+h)-\sigma(x)|\leq \ell_2(x)|h|^\gamma.
\end{align}
\end{enumerate}
\begin{enumerate}[{\bf (H$^{r,\KK}_{\rm glo}$)}]
\item For given $r>-\alpha$, there are $\eps_0\in(0,1]$ small enough and $c_0, c_1>0$ such that for all $|x|>1$,
\begin{align}\label{BB1}
\<x,b(x)\>+\eps_0\ell_1(x)|x|+\KK (\|\sigma(x)\|+\eps_0\ell_2(x))^\alpha|x|^{2-\alpha}\leq -c_0|x|^{2+r}+c_1,
\end{align}
where $\KK$ is defined by \eqref{KK} below.
\end{enumerate}
\br\rm\label{Re8}
Notice that \eqref{Hol} is equivalent to the locally $\gamma$-H\"older continuity of $b$ and $\sigma$.
If $b,\sigma$ are locally Lipschitz, i.e. $\gamma=1$ in \eqref{Hol}, then one can take $\eps_0=0$ in \eqref{BB1}. When $\gamma<1$, in order to adopt the stopping time to localize the coefficients, we need to use the pathwise uniqueness 
of strong solutions.
Thus we must mollify the coefficients and require $\eps_0>0$ in \eqref{BB1} to have some uniform estimates for the approximation coefficients.
Examples of $b,\sigma$ satisfying {\bf (H$^{r,\KK}_{\rm glo}$)} shall be provided below. In particular, $b$ and $\sigma$ are allowed to be polynomial growth. It is interesting that when $\alpha\in(0,1)$, $\sigma$ can grow faster than $b$ (see Example \ref{Ex13} below).
\er

Under the above assumptions, it is well known that there is a unique weak solution to SDE \eqref{sde:main} (see \cite{CZZ21}).
The semigroup associated with the Markov process $X_t(x)$ is then defined by
$$
P_tf(x):=\mathbb{E}f(X_t(x)),\ \ f\in C_b(\mR^d).
$$
By It\^o's formula, one sees that the generator of $P_t$ is given by
\begin{align}\label{gen}
\mathscr{L}f(x):=\mathcal{L}_\sigma f(x)+\langle b(x), \nabla f(x)\rangle,
\end{align}
with
\begin{align}\label{CL1}
\mathcal{L}_\sigma f(x)=\frac{\alpha 2^\alpha\Gamma(\frac{d+\alpha}{2})}{\Gamma(\frac d2)\Gamma(\frac{2-\alpha}2)}\int_{\mathbb{R}^d}[f(x+\sigma(x)z)+f(x-\sigma(x)z)-2f(x)]\dif z/|z|^{d+\alpha}.
\end{align}
In particular, 
\begin{align}\label{CL11}
\Delta^{\frac\alpha2} f(x)=\frac{\alpha 2^\alpha\Gamma(\frac{d+\alpha}{2})}{\Gamma(\frac d2)\Gamma(\frac{2-\alpha}2)}\int_{\mathbb{R}^d}[f(x+z)+f(x-z)-2f(x)]\dif z/|z|^{d+\alpha}.
\end{align}
We call SDE \eqref{sde:main} being supercritical for $\alpha\in(0,1)$ since in this case, from the view point of PDE, the drift term plays a dominant role compared with 
the diffusion term $\cL_\sigma$ (see \cite{CZZ21}).

\medskip

Our main result of this paper is 
\bt\label{Main}
Let $r>-\alpha$ and $p\in((-r)\vee 0,\alpha)$. Define $V_p(x):=1+|x|^p$. Suppose that {\bf (H$_{\rm loc}$)} and {\bf (H$^{r,\KK}_{\rm glo}$)} hold.
Then there is a unique invariant probability measure $\mu$ associated with $(P_t)_{t\geq 0}$. Moreover, we have

(i) If $r\geq 0$, then there are $\lambda, C_p>0$ such that for all $t>0$ and $x\in\mR^d$,
$$
\sup_{\|\varphi/V_p\|_\infty\leq 1}|P_t\varphi(x)-\mu(\varphi)|\leq C_p\e^{-\lambda t}V_p(x).
$$

(ii) If $r>0$, then there are $\lambda, C>0$ such that for all $t>0$,
$$
\sup_{x\in\mR^d}\|P_t(x.\cdot)-\mu\|_{\rm Var}\leq C\e^{-\lambda t}.
$$
\et

Below we provide two examples to illustrate our main result.
\bx\rm\label{Ex13}
The following example shows that $b$ and $\sigma$ can be polynomial decay and growth:
$$
b(x)=-x(1+|x|^2)^{\beta/2},\ \ \sigma(x)=(1+|x|^2)^{\gamma/2}\mI,\ \ \beta\in(-\alpha,\infty),\ \gamma\in(-\infty,1+\tfrac\beta\alpha).
$$
For $|x|>1$, by Young's inequality, we have
\begin{align*}
&\<x,b(x)\> +\KK \|\sigma(x)\|^\alpha|x|^{2-\alpha}=-|x|^2(1+|x|^2)^{\frac\beta2}+\KK(1+|x|^2)^{\frac{\alpha\gamma}2}|x|^{2-\alpha}\\ 
&\qquad\leq-2^{\frac\beta 2\wedge 0}|x|^{2+\beta}+2^{\frac {\alpha(\gamma\vee 0)}{2}}q|x|^{2-\alpha+\alpha(\gamma\vee 0)}\leq -2^{(\frac\beta 2\wedge 0)-1}|x|^{2+\beta}+c_1.
\end{align*}
In particular, when $\alpha>1$, $\beta$ can be less than $-1$ so that $b$ is polynomial decay. Moreover, for fixed $\beta>0$, when $\alpha\downarrow 0$, $\gamma$ can go to infinity.
In other words, $\sigma$ can grow faster than $b$. This is not surprise since one can image that for the SDE driven by compound Poisson processes (in some sense, corresponding to $\alpha=0$), $\sigma$ can be
arbitrarily growth.
\ex
\bx\rm
The following example shows that $b$ and $\sigma$ can even be exponentially growth:
$$
b(x)=-x\e^{|x|},\ \ \sigma(x)=\e^{|x|/\beta}\mI,\ \ \beta\in(\alpha,2).
$$
It is easy to see that for $|h|\leq 1$,
$$
|b(x+h)-b(x)|\leq (1+|x|)\e^{|x|+1}|h|=:\ell_1(x)|h|,
$$
and
$$
|\sigma(x+h)-\sigma(x)|\leq \e^{(|x|+1)/\beta}|h|=:\ell_2(x)|h|.
$$
Thus for any $|x|>1$ and $q>1$, by Young's inequality, one can choose $\eps_0$ small enough so that
\begin{align*}
&\<x,b(x)\>+\eps_0\ell_1(x)|x|+\KK (\|\sigma(x)\|+\eps_0\ell_2(x))^\alpha|x|^{2-\alpha}\\
&\quad=-|x|^2\e^{|x|}+\eps_0(1+|x|)|x|\e^{|x|+1}+\KK(\e^{|x|/\beta}+\eps_0\e^{(|x|+1)/\beta})^\alpha|x|^{2-\alpha}\\ 
&\quad\leq-|x|^2\e^{|x|}+\tfrac14|x|^2\e^{|x|+1}+2\KK\e^{\alpha|x|/\beta}|x|^{2-\alpha}\leq -\tfrac12|x|^2\e^{|x|}+c_1.
\end{align*}
Still, when $\alpha\downarrow 0$, the diffusion $\sigma$ can grow faster than drift $b$.
\ex

When the driven noise is Brownian motion, 
the ergodicity of SDEs \eqref{sde:main} has been studied extensively in the literature \cite{BGL14, Cer01, Da92,  W06,  XZ20}, etc., 
and the dissipativity condition \eqref{BB1} 
usually takes the following form (for example, see \cite[(7.1)]{XZ20}):
$$
2\<x,b(x)\>+\|\sigma(x)\|^2\leq-c_0|x|^2+c_1.
$$
In recent years, there also has been a great interest in the study of the ergodicity for SDEs driven by L\'evy processes, see \cite{K09,LMW21,Wang16,XZ20} and the references therein. 
To the best of the authors' knowledge, there exist at least three well developed methods for studying the exponential ergodicity of a Markov process defined through an SDE: 
Coupling method (cf. \cite{HMW21, K09, LW20, LMW21, LW19, M17}), Functional inequality method (see \cite{BGL14, C05, W06}), 
Meyn-Tweedie's minorization condition (cf. \cite{GM06,M07, MSH02, MT09}). 
Each method has its own merits and may be adapted to different situations.  For the coupling method, in order to construct a successful coupling it is usually assumed that the following  one-sided monotone 
condition holds (see \cite{LW19}): for some $R>0$,
\begin{equation}
\langle  x-y, \nabla U(x)-\nabla U(y)\rangle\geq
\begin{cases}
 -L|x-y|^2,\quad & |x-y|<R,\\
K|x-y|^2, \quad  & |x-y|\geq R.
\end{cases}
\label{dri:dis}
\end{equation}
While, for the functional inequality method, it is usually assumed that the potential function $U$ is at least $C^2$-smooth (see \cite{BGL14, W06}). 
The obvious advantage of this method is that the quantitative convergence rate can be calculated from the parameters.
Finally, Meyn-Tweedie's approach is based on the Lyapunov estimates together with verifying the strong Feller property and irreducibility for the associated semigroup.

\smallskip

To prove our main theorem, we adopt the Meyn-Tweedie method. As we know,
the strong Feller property reflects the regularization effects of the noise, which is usually related to the continuity of the transition density of a Markov process with respect to the starting point.
There are many ways to establish the  strong Feller property, e.g., the gradient estimates, Harnack's inequalities, etc.. For the irreducibility, when the driving noise is Brownian motion, there are several ways for choices: 
by solving  a control problem \cite{MSH02}, by Girsanov's transformation \cite{XZ20}, by the classical supported theorem of SDE \cite{C07}. 
However, when the driven noise is an $\alpha$-stable noise, the above method for irreducibility does not work any more.
In \cite{XZ20}, we directly verify the positivity of the Dirichlet heat kernel by localization method. Therein, it is restricted to $\alpha\in(1,2)$ since in this case, 
the heat kernel of the operator \eqref{gen} is comparable with the one of the fractional Laplacian $\Delta^{\frac\alpha 2}$.
In a recent work \cite{MZ20}, when the conditions in {\bf (H$_{\rm loc}$)} holds globally, i.e., the constant $C_m$ does not depend on $m$ and $\ell_1(x),\ell_2(x)$ are constants, 
we obtain the following two-sided estimates for the density of SDE \eqref{sde:main}: for any $T>0$ and some $C>1$, and for all $t\in(0,T]$ and $x,y\in\mR^d$,
\begin{align}
\label{2sest:ubd}
C^{-1} t(t^{\frac{1}{\alpha}}+|\theta_t(x)-y|)^{-d-\alpha}\leq p(t,x,y)\leq C t(t^{\frac{1}{\alpha}}+|\theta_t(x)-y|)^{-d-\alpha},
\end{align}
where $\theta_t(x)$ solves the following regularized ODE:
\begin{align}\label{TT9}
\dot\theta_t(x)=b*\phi_{t^{1/\alpha}}(\theta_t(x)),\ \ \theta_0(x)=x,
\end{align}
where $(\phi_\eps)_{\eps>0}$ is a family of standard mollifiers. The above estimates still allow us to show the irreducibility of $P_t$ by localization.
For the existence of invariant measures, we use the standard Krylov-Bogoliubov's theorem by showing some Lyapunov's type estimate.

\smallskip

The paper is organized as follows. In Section 2, we recall some basic notions about the ergodicity of Markov processes. These notions are standard in literature (see \cite{Da92,Rev99}). Moreover,
we also prove a general criterion for the irreducibility of a Markov process in terms of the heat kernel estimates, which is motivated by the works of \cite{Chu95,XZ20}. 
In Section 3, we prove our main result Theorem \ref{Main}. For the existence of invariant measures, by the classical Krylov-Bogoliubov's  method, we show the Lyapunov type estimates for $X_t(x)$.
For the uniqueness of invariant measure as well as the exponential ergodicity, we show the strong Feller property and the irreducibility of $P_t$ by the heat kernel estimates 
and suitable localization and smoothing arguments. In Section 4, we present an application of our main result to the heavy-tailed sampling. In Appendix, the proof of weak convergence for mollifying SDEs
is provided for readers' convenience.

\section{Preliminaries}

In this section we prepare some notions and criterion about the ergodicity and irreducibility of general Markov processes in Euclidean spaces. 
Fix $N\in\mN$. 
Let $\mD$ be the space of all c\'adl\'ag functions from $[0,\infty)$ to $\mR^N$, which is endowed with the Skorokhod topology so that $\mD$ becomes a Polish space. Let $X_t(\omega):=\omega_t$ be the coordinate process over $\mD$. For $s>0$, let $\theta_s:\mD\to\mD$ be the shift operator:
$$
\theta_s(\omega)(t)=\omega(t+s),\ \ t\geq 0.
$$
Let $(\mP_x)_{x\in\mR^N}$ be a family of probability measures over $\mD$ so that $\sM:=\{(X_t)_{t\geq 0},(\mP_{x})_{x\in\mR^N}\}$ forms a family of Markov processes with regard to the natural filtration $(\sF_t)_{t\geq 0}$.
More precisely,

(i) For each $x\in\mR^N$, $\mP_x(X_0=x)=1$.

(ii) For each $A\in\sB(\mD)$, $x\mapsto\mP_x(A)$ is Borel measurable.

(iii) For any $s<t$ and $A\in\sB(\mD)$, it holds that
\begin{align}\label{Ma1}
\mP_x(A\circ \theta_s|\sF_s)=\mP_{X_s}(A).
\end{align}
Let $\cB_b(\mR^N)$ be the space of all bounded measurable functions over $\mR^N$.  For $t\geq 0$ and $f\in\cB_b(\mR^N)$, the semigroup associated with $\sM$ is defined by
$$
P_tf(x):=\mE_x f(X_t),\ \ x\in\mR^N,
$$
where $\mE_x$ stands for the expectation with respect to $\mP_x$. 

The following notions are standard and can be found in \cite{Da92,MT09}.
\begin{enumerate}[$\bullet$]
\item A probability measure $\mu\in\cP(\mR^N)$ is called an invariant probability measure of $(P_t)_{t\geq 0}$, if
$$
\mu(P_tf)=\mu(f),\ \ \forall t\geq 0,\ f\in\cB_b(\mR^N). 
$$
\item We call $(P_t)_{t\geq 0}$ ergodic, if there exists a unique invariant probability measure $\mu$ of $(P_t)_{t\geq 0}$, equivalently, for any $x\in\mR^N$ and $f\in\cB_b(\mR^N)$,
$$
\lim_{t\to\infty}\frac{1}{t}\int^t_0P_sf(x)\dif s=\mu(f).
$$
\item Let $V:\mR^N\to[1,\infty)$ be a measurable function. We call $(P_t)_{t\geq 0}$ be $V$-uniformly exponential ergodic if there are constants $C,\lambda>0$ 
and an invariant probability measure $\mu$ such that
$$
\sup_{\|\phi/V\|_\infty\leq 1}|P_t\phi(x)-\mu(\phi)|\leq C\e^{-\lambda t} V(x),\ \ t\geq 0,\ x\in\mR^N.
$$
In particular, if $V\equiv1$, then we call $(P_t)_{t\geq 0}$ be uniformly exponentially ergodic.
\item $(P_t)_{t\geq 0}$ is said to be strong Feller, if $P_tf\in C_b(\mathbb{R}^N)$ for any $t>0$ and $f\in\mathcal{B}_b(\mathbb{R}^N)$.
\item $(P_t)_{t\geq 0}$ is called irreducible, if for any $t>0$, open ball $B$ and $x\in\mathbb{R}^N$, $P_t\mathbf{1}_B(x)>0$.
\end{enumerate}

The following general result is taken from \cite[Theorem 2.5]{GM06}.
\bt\label{Th11}
Suppose that $(P_t)_{t\geq 0}$ is strong Feller and irreducible. If there are $p,\lambda>0$ and $C_1, C_2>0$ such that
$$
\mE_x|X_t|^p\leq C_1|x|^p\e^{-\lambda t}+C_2,\ \   t\geq 0,\ x\in\mR^N,
$$
then $(P_t)_{t\geq 0}$ is $V$-uniformly ergodic with $V(x)=1+|x|^p$. If there are $p, T_0, C_3>0$ such that
$$
\mE_x|X_t|^p\leq C_3,\ \   t\geq T_0,\ x\in\mR^N,
$$
then $(P_t)_{t\geq 0}$ is uniformly exponentially ergodic.
\et

Next we present a general criterion for irreducibility in terms of the heat kernel estimates. We believe that it could be used to other cases.

Let $\rho: (0,\infty)\times(0,\infty)\to(0,\infty)$ be a continuous function with the properties that
for any $t>0$, $r\mapsto \rho(t,r)$ is decreasing on $(0,\infty)$, and
for any $0<\delta<T<\infty$, 
\begin{align}\label{DD4}
\sup_{t\in[\delta,T]}\sup_{r>0}\rho(t,r)<\infty.
\end{align}
In the following we make the following assumptions about the Markov process $\sM$:
\begin{enumerate}[{\bf (A$_1$)}]
\item[{\bf (A$_1$)}] Suppose that for each $t>0$, with respect to $\mP_x$, $X_t$ admits a family of transition probability density function $p(t,x,y)$ in $\mR^N$ so that
\begin{align}\label{DD1}
\mathbb{P}_x(X_t\in A)=\int_Ap(t,x,y)\dif y,\ \ \forall A\in\sB(\mR^N),\ t>0,\ x\in\mR^N.
\end{align}
Moreover, $t\mapsto X_t$ is stochastically continuous, equivalently, 
\begin{align}\label{DD11}
\mP_x(X_t\not= X_{t-})=0,\ \ \forall t>0, x\in\mR^N.
\end{align}

\item[{\bf (A$_2$)}] We suppose that for any $T>0$, there are constants $\lambda_0, C_0\geq 1$ such that for any $t\in(0,T]$ and $x\in\mR^N$,
and for Lebesgue almost all $y\in\mR^N$,
\begin{align}\label{DD2}
C_{0}^{-1}\rho(t,\lambda_0\Gamma_t(x,y))\leq p(t,x,y)\leq C_{0} \rho(t,\Gamma_t(x,y)/\lambda_0),
\end{align}
where $\Gamma_t(x,y):(0,\infty)\times\mR^N\times\mR^N\to(0,\infty)$ is a measurable function, and for each $y\in\mR^N$, 
$$
\mbox{$(t,x)\mapsto \Gamma_t(x,y)$ is continuous.}
$$
\end{enumerate}
\br\rm
In {\bf (A$_2$)}, for the standard Gaussian case, one usually takes $\rho(t,r)=t^{-d/2}\e^{-r^2/t}$ and $\Gamma_t(x,y)=|x-y|$.
For the $\alpha$-stable case, if $\alpha\in[1,2)$, one takes $\rho(t,r)=t(t^{1/\alpha}+r)^{-d-\alpha}$ and $\Gamma_t(x,y)=|x-y|$;
and if $\alpha\in(0,1)$, one  takes $\Gamma_t(x,y)=|x-\theta_t(y)|$, where $\theta_t(y)$ is the ODE flow associated with drift $b$ (see \eqref{TT9} above).
We would like to emphasize that we do not make any continuity assumption about $y\mapsto p(t,x,y)$. Usually, the regularity of the density requires more regular coefficients
in the theory of SDEs.
\er

Let $D$ be an open subset of $\mathbb{R}^N$.  The exit time of $X_t$ from $D$ is defined by
$$
\tau_D:=\inf\big\{t>0,X_t\notin D\big\}.
$$

We have the following general result.
\bt\label{41}
Let $D_0\Subset D$ be two connected  domains of $\mR^N$ so that $D$ has Lebesgue-zero measure boundary. Suppose that {\bf (A$_1$)} and {\bf (A$_2$)} hold, 
and the above $\Gamma_t(x,y)$ and $\rho(t,r)$ satisfy the following conditions: 
there are $t_0,\delta_0>0$ small enough and functions $\ell_0(t)$, $\ell_1(t)$ defined on $(0,t_0]$ such that for all $t\in(0,t_0]$, $y_0\in D_0$, $x,y\in B_{\delta_0}(y_0)\subset D_0$ and $x'\notin D$,
\begin{align}\label{AA01}
\Gamma_{t}(x,y)\leq\ell_0(t),\ \Gamma_{t}(x',y)\geq\ell_1(t),
\end{align}
and for the $C_0,\lambda_0$ in \eqref{DD2}, $t\mapsto\rho(t,\ell_1(t)/\lambda_0)$ is increasing on $(0,t_0]$, and 
\begin{align}\label{AA2}
C_0^{-1}\rho\big(t,\lambda_0\ell_0(t)\big)\geq 2C_0\rho\big(t,\ell_1(t)/\lambda_0).
\end{align}
Then for any $x_0,y_0\in D_0$ and $r_1,r_2<{\rm dist}(D_0,D^c)$,
\begin{align}\label{kill:sp}
\inf_{x\in B_{r_1}(x_0)}\mP_x\Big(X_t\in B_{r_2}(y_0); t<\tau_D\Big)\geq c_0>0,
\end{align}
where $c_0$ only depends on $\rho,\lambda_0,  C_0, t_0, \delta_0$ and $r_1,r_2$, $x_0,y_0, D_0, D$.
\et
\begin{proof}
Let $t_0$ and $\delta_0$ be as in the assumptions. We divide the proofs into two steps.

(i) In this step we show that for all $y_0\in D_0$, $\delta\in(0,\frac{\delta_0}2]$ and $t\in(0,t_0]$,
\begin{align}\label{kill:sp1}
\inf_{x\in B_{2\delta}(y_0)}\mP_x\Big(X_t\in B_{\delta}(y_0); t<\tau_D\Big)\geq C_0\rho\big(t,\ell_1(t)/\lambda_0\big)|B_{\delta}(y_0)|>0.
\end{align}
Fix $s\in(0,t)$ and $n\in\mN$. We define a stopping time $T_n$ as follows:
\begin{equation*}
T_n=
\begin{cases}
ks2^{-n}\quad &\text{if}\ (k-1)s2^{-n}\leq \tau_D<ks2^{-n},\\
\infty\quad &\text{if}\ \tau_D\geq s.
\end{cases}
\end{equation*}
Clearly, on $\{\tau_D<s\}$,
$$
s\geq T_n\downarrow\tau_D\ \ \mbox{as $n\to\infty$.}
$$
For any bounded $A\in\sB(D)$, by the Markov property  and \eqref{DD1},   we have
\begin{align*}
\mathbb{P}_x(X_t\in A;\tau_D<s)
&=\sum^{2^n}_{k=1}\mathbb{P}_x(X_t\in A;(k-1)s2^{-n}\leq \tau_D<ks2^{-n})\nonumber\\
&=\sum^{2^n}_{k=1}\mathbb{P}_x(X_t\in A; T_n=ks2^{-n})\nonumber\\
&=\sum^{2^n}_{k=1}\mathbb{E}_x\Big(\mP_{X_{T_n}}(X_{t-T_n}\in A);T_n=ks2^{-n}\Big)\no\\
&=\sum^{2^n}_{k=1}\mathbb{E}_x\left(\int_Ap(t-T_n,X_{T_n},y)\dif y;T_n=ks2^{-n}\right)\nonumber\\
&=\mathbb{E}_x\left(\int_Ap(t-T_n,X_{T_n},y)\dif y;\tau_D<s\right)\\
&\stackrel{\eqref{DD2}}{\leq} C_0\mathbb{E}_x\left(\int_A\rho\big(t-T_n,\Gamma_{t-T_n}(X_{T_n},y)/\lambda_0\big)\dif y;\tau_D<s\right).
\end{align*}
Letting $n\to\infty$ and by \eqref{DD4} and the dominated convergence theorem, we get
$$
\mathbb{P}_x(X_t\in A;\tau_D<s)\leq C_0\mathbb{E}_x\left(\int_A\rho\big(t-\tau_D,\Gamma_{t-\tau_D}(X_{\tau_D},y)/\lambda_0\big)\dif y;\tau_D<s\right).
$$
Below we take $A=B_{\delta}(y_0)$. For $y\in B_\delta(y_0)$, since $X_{\tau_D}\notin D$ and
$r\to \rho(t-\tau_D,r)$ is  decreasing,  $t\mapsto\rho(t,\ell_1(t)/\lambda_0)$ is increasing, 
by \eqref{AA01} we have
$$
\rho\big(t-\tau_D,\Gamma_{t-\tau_D}(X_{\tau_D},y)/\lambda_0\big)\leq\rho\big(t-\tau_D,\ell_1(t-\tau_D)/\lambda_0\big)\leq \rho\big(t,\ell_1(t)/\lambda_0\big).
$$
Hence, for any $s<t$,
\begin{align*}
\mathbb{P}_x\(X_t\in B_{\delta}(y_0);\tau_D<s\)\leq C_0\int_{B_{\delta}(y_0)}\rho\big(t,\ell_1(t)/\lambda_0\big)\dif y.
\end{align*}
On the other hand, by \eqref{DD11} and \eqref{DD2}, we have
\begin{align*}
\mP_x(\tau_D=t)&=\mP_x(\tau_D=t, X_t\in\p D)+\mP_x(\tau_D=t, X_t\notin\p D, X_s\in D, s<\tau_D)\\
&\leq\mP_x(X_t\in\p D)+\mP_x(X_t\not=X_{t-})=0.
\end{align*}
Thus, by the lower bound in \eqref{DD2} and the above estimates, we arrive at
\begin{align*}
\mathbb{P}_x\(X_t\in B_{\delta}(y_0);t<\tau_D\)&=\mathbb{P}_x\(X_t\in B_{\delta}(y_0)\)-\mathbb{P}_x\(X_t\in B_{\delta}(y_0);\tau_D<t\)\\
&=\int_{B_{\delta}(y_0)}p(t,x,y)\dif y-\lim_{s\uparrow t}\mathbb{P}_x\(X_t\in B_{\delta}(y_0);\tau_D<s\)\\
&\geq \int_{ B_{\delta}(y_0)} \Big(C_0^{-1} \rho\big(t,\lambda_0\Gamma_t(x,y)\big)-C_0\rho\big(t,\ell_1(t)/\lambda_0\big)\Big)\dif y.
\end{align*}
For $x\in B_{2\delta}(y_0), y\in B_\delta(y_0)$, since $r\to \rho(t,r)$ is decreasing, by \eqref{AA01}, we obtain
\begin{align*}
\mathbb{P}_x\(X_t\in B_{\delta}(y_0);t<\tau_D\)&\geq \Big(C_0^{-1}\rho\big(t,\lambda_0\ell_0(t)\big)-C_0\rho\big(t,\ell_1(t)/\lambda_0\big)\Big)|B_{\delta}(y_0)|\\
&\!\!\!\stackrel{\eqref{AA2}}{\geq}C_0\rho\big(t,\ell_1(t)/\lambda_0\big)|B_{\delta}(y_0)|.
\end{align*}
Thus we get \eqref{kill:sp1}.

(ii) By \eqref{DD1}, there is a density $p^D(t,x,y)$ so that for each $A\in\sB(D)$,
$$
\int_A p^D(t,x,y)\dif y=\mathbb{P}_x(X_t\in A;t<\tau_D).
$$
Moreover, by the Markov property, we have for each $0<s<t$,
\begin{align}\label{TT8}
\int_Ap^D(t,x,y)\dif y=\int_Dp^D(s,x,z)\dif z\int_Ap^D(t-s,z,y)\dif y.
\end{align}
Indeed, by $s+\tau_D\circ\theta_s=\tau_D$ on $\{s<\tau_D\}$ and \eqref{Ma1},
\begin{align*}
\mathbb{P}_x(X_t\in A;t<\tau_D)&=\mE_x\Big(\mathbb{P}_x\big(X_t\in A;t<\tau_D|\sF_s\big)\Big)\\
&=\mE_x\Big(\mathbb{P}_x\big(X_{t-s}\circ \theta_s\in A;t<s+\tau_D\circ\theta_s, s<\tau_D|\sF_s\big)\Big)\\
&=\mE_x\Big(\mathbb{P}_{X_s}\big(X_{t-s}\in A;t-s<\tau_D\big); s<\tau_D\Big)\\
&=\int_Dp^D(s,x,z)\int_Ap^D(t-s,z,y)\dif y\dif z.
\end{align*}
Let $x_0\in D_0$. Since $D_0$ is connected, there is a continuous curve $\cC\subset D_0$ connecting $x_0$ and $y_0$. Let
$$
\delta:=({\rm dist}(D_0,\p D)\wedge\delta_0)/4.
$$ 
It is easy to see that there exists $n\in\mathbb{N}$ large enough and $\{x_i,i=0,1,\cdots,n+1\}\subset\cC$ such that 
$$
t\leq nt_0,\ \ x_i\in B_{\delta}(x_{i-1}),\ i=1,\cdots,n+1,\ B_{\delta}(x_{n+1})\subset B_{r_2}(y_0).
$$
For each $j=1,\cdots,n+1$, since $B_\delta(x_{j-1})\subset B_{2\delta}(x_j)$, by what we have proved in Step (i), 
\begin{align*}
\inf_{y_{j-1}\in B_{\delta}(x_{j-1})}\int_{B_{\delta}(x_j)} p^D(\tfrac{t}n, y_{j-1},y_j)\dif y_j
&\geq\inf_{y_{j-1}\in B_{2\delta}(x_{j})}\int_{B_{\delta}(x_j)} p^D(\tfrac{t}n, y_{j-1},y_j)\dif y_j\\
&\geq C_0\rho\big(\tfrac tn,\ell_1(\tfrac tn)/\lambda_0\big)|B_{\delta}(y_0)|>0.
\end{align*}
Hence, for all $x\in B_{\delta}(x_0)$, by \eqref{TT8},
\begin{align*}
&\int_{B_{r_2}(y_0)}p^D(t,x,y)\dif y=\int_D p^D(\tfrac{t}n,x,y_1)\dif y_1\cdots\int_D p^D(\tfrac{t}n, y_{n-1},y_n)\dif y_n\int_{B_{r_2}(y_0)}p^D(\tfrac{t}{n},y_n,y)\dif y\\
&\qquad\geq\int_{B_{\delta}(x_1)} p^D(\tfrac{t}n,x,y_1)\dif y_1\cdots\int_{B_{\delta}(x_n)} p^D(\tfrac{t}n, y_{n-1},y_n)\dif y_n\int_{B_{\delta}(x_{n+1})}p^D(\tfrac{t}{n},y_n,y)\dif y\\
&\qquad\geq \Big(C_0\rho\big(\tfrac tn,\ell_1(\tfrac tn)/\lambda_0\big)|B_{\delta}(y_0)|\Big)^{n+1}.
\end{align*}
Thus,
$$
\inf_{x\in B_{\delta}(x_0)}\mP_x\Big(X_t\in B_{r_2}(y_0); t<\tau_D\Big)\geq \Big(C_0\rho\big(\tfrac tn,\ell_1(\tfrac tn)/\lambda_0\big)|B_{\delta}(y_0)|\Big)^{n+1}>0.
$$
For general $r_1\in(0,{\rm dist}(D_0,D))$, it follows by the finitely covering technique.
\end{proof}

\section{Proof of Theorem \ref{Main}}

In this section we prove Theorem \ref{Main} by verifying the conditions in Theorem \ref{Th11}.
Below we fix 
$$
r>-\alpha,\ \ p\in((-r)\vee 0,\alpha),
$$ 
and define
\begin{align}\label{KK}
V_p(x):=(1+|x|^2)^{p/2},\quad \KK:=\ff{2^{4+2\alpha}\pi^{\frac{d}{2}}\Gamma(\ff{d+\alpha}{2})}{\Gamma(\ff{d}{2})^2\Gamma(\ff{2-\alpha}{2})}\left(\ff{\alpha}{2-\alpha}+\ff{\alpha}{(\alpha-p)p}\right).
\end{align}
First of all, we show that $V_p$ is a Lyapunov function of the operator $\sL$.
\bl\label{Le31}
Suppose that $b,\sigma$ are locally bounded and for some $c_0,c_1>0$,
\begin{align}\label{LL1}
\KK\|\sigma(x)\|^\alpha|x|^{2-\alpha}+\<x,b(x)\>\leq -c_0|x|^{2+r}+c_1,\ \ |x|>1.
\end{align}
Then there are $\kappa_0, \kappa_1>0$ such that for all $x\in\mR^d$,
\begin{align}\label{Lya}
\sL V_p(x)=\cL_\sigma V_p(x)+\<b(x),\nabla V_p(x)\>\leq -\kappa_0 V_p(x)^{1+\frac rp}+\kappa_1.
\end{align}
\el
\begin{proof}
It suffices to prove \eqref{Lya} for $|x|>1$. By \eqref{CL1}, we make the following decomposition:
\begin{align*}
c^{-1}_{d,\alpha}\cL_\sigma V_p(x)
&=\int_{|z|\leq \frac{|x|}{2\|\sigma(x)\|}}\Big[V_p(x+\sigma(x)z)+V_p(x-\sigma(x)z)-2V_p(x)\Big]\frac{\dif z}{|z|^{d+\alpha}}\\
&+\int_{|z|>\frac{|x|}{2\|\sigma(x)\|}}\Big[V_p(x+\sigma(x)z)+V_p(x-\sigma(x)z)-2V_p(x)\Big]\frac{\dif z}{|z|^{d+\alpha}}=:I_1+I_2,
\end{align*}
where $c_{d,\alpha}:=\frac{\alpha 2^\alpha\Gamma(\frac{d+\alpha}{2})}{\Gamma(\frac d2)\Gamma(\frac{2-\alpha}2)}$.
For $I_1$, denoting by $A(x):=(1+|x|^2)\mI-(2-p)x\otimes x$, we have
$$
\nabla^2 V_p(x)=p(1+|x|^2)^{\frac{p}{2}-2}A(x).
$$
Noting that for any $\xi\in\mathbb{R}^d$,
\begin{align*}
|\langle\xi, A(x)\xi\rangle|\leq (1+|x|^2)|\xi|^2+(2-p)|\langle\xi, x\rangle|^2\leq 3(1+|x|^2)|\xi|^2.
\end{align*}
By Taylor's expansion, for  $|z|\leq \frac{|x|}{2\|\sigma(x)\|}$, we have
\begin{align*}
&|V_p(x+\sigma(x)z)+V_p(x-\sigma(x)z)-2V_p(x)|\\
&\quad\leq 3p|\sigma(x)z|^2\int^1_0\int^1_{-1}(1+|x+s_1s_2\sigma(x)z|^2)^{\frac p2-1}\dif r_1\dif r_2\\
&\quad\leq 3p\|\sigma(x)\|^2|z|^2\Big(1+\tfrac{|x|^2}{4}\Big)^{\frac p2-1}.
\end{align*}
Hence, for $S_d:=2\pi^{\ff{d}{2}}/\Gamma(\ff{d}{2})$,
\begin{align*}
I_1&\leq 3p \|\sigma(x)\|^2\Big(1+\tfrac{|x|^2}{4}\Big)^{\frac p2-1}\int_{|z|\leq \frac{|x|}{2\|\sigma(x)\|}}|z|^{2-d-\alpha}\dif z\\
     &=3pS_d \|\sigma(x)\|^2\Big(1+\tfrac{|x|^2}{4}\Big)^{\frac p2-1}\int_0^{\frac{|x|}{2\|\sigma(x)\|}}r^{1-\alpha}\dif r\\
     &\leq \frac{p2^{2+\alpha} S_d}{2-\alpha}\|\sigma(x)\|^\alpha|x|^{p-\alpha}.
\end{align*}
For $I_2$, noting that
$$
|(1+|x+z|^2)^{\frac p2}-(1+|x|^2)^{\frac p2}|\leq ||x+z|^2-|x|^2|^{\ff{p}{2}}\leq 2^{\ff{p}{2}}|x|^{\ff{p}{2}}|z|^{\ff{p}{2}}+|z|^p,
$$
we have 
\begin{align*}
I_2\leq 2\int_{|z|>\frac{|x|}{2\|\sigma(x)\|}}\Big[2^{\ff{p}{2}}|x|^{\ff{p}{2}}|\sigma(x)z|^{\ff{p}{2}}+|\sigma(x)z|^p\Big]\frac{\dif z}{|z|^{d+\alpha}}
     \leq \frac{2^{3+\alpha} S_d}{\alpha-p}\|\sigma(x)\|^\alpha|x|^{p-\alpha}.
\end{align*}
Combining the above calculations, we get for $|x|>1$, 
\begin{align*}
\cL_\sigma V_p(x)\leq& 2^{3+\alpha}p S_d c_{d,\alpha}\left(\frac{1}{2-\alpha}+\frac{1}{(\alpha-p)p}\right)\|\sigma(x)\|^\alpha|x|^{p-\alpha}\\
\leq& p2^{\frac p2-1}\KK \|\sigma(x)\|^\alpha|x|^{p-\alpha}
\end{align*}
Thus, then  by \eqref{LL1} and $\nabla V_p(x)=p x(1+|x|^2)^{\frac{p}{2}-1}$, we have
\begin{align*}
\cL_\sigma V_p(x)+\<b(x),\nabla V_p(x)\>&\leq p2^{\frac p2-1}\KK \|\sigma(x)\|^\alpha|x|^{p-\alpha}+p(1+|x|^2)^{\frac p2-1}\<x,b(x)\>\\
&\leq p(1+|x|^2)^{\frac p2-1}\Big(\KK\|\sigma(x)\|^\alpha|x|^{2-\alpha}+\<x,b(x)\>\Big)\\
&\leq p(1+|x|^2)^{\frac p2-1}\big(-c_0|x|^{2+r}+c_1\big)\\
&\leq -\kappa_0 (1+|x|^2)^{\frac {p+r}2}+\kappa_1=-\kappa_0V_p(x)^{1+\frac {r}p}+\kappa_1,
\end{align*}
where in the second inequality we use $(1+|x|^2)^{1-\frac p2}\leq 2^{1-\frac p2}|x|^{2-p}$ with $|x|>1$ , and $\kappa_0=pc_0$.
\end{proof}

As a consequence of the above Lyapunov-type estimate, we have the following result (see \cite[Lemma 7.1]{XZ20}).
\begin{lemma}
Let $X_t(x)$ be any weak solution of SDE \eqref{sde:main} starting from $x$.
Under the assumptions of Lemma \ref{Le31},  there exists a constant $C>0$ such that for all $x\in\mR^d$ and $t>0$,
\begin{align}\label{e:x1}
\left[\mathbb{E}\left(\sup_{s\in[0,t]}V_p(X_s(x))^{\frac12}\right)\right]^2+\mE\left(\int^{t}_0 V_p(X_s(x))^{1+\frac rp}\dif s\right)\lesssim_C V_p(x)+t,
\end{align}
and for the $\kappa_0$ in \eqref{Lya},
\begin{align}\label{e:x}
\mathbb{E}V_p(X_t(x))\leq 
\left\{
\begin{aligned}
&\e^{-\kappa_0t}V_p(x)+C,\ \ &r=0,\\
&C(1+t^{-p/r}),\ \ &r>0. 
\end{aligned}
\right.
\end{align}
\end{lemma}
\begin{proof}
Let $N(t,\dif z)$ be the Poisson random measure associated with $L^\alpha_t$, i.e.,
$$
N(t,\Gamma):=\sum_{s\in(0,t]}\1_{\Gamma}(L^\alpha_s-L^\alpha_{s-}),\ \ \Gamma\in\sB(\mR^d).
$$
Let ${\wt N}(t,\dif z):=N(t,\dif z)-t|z|^{-d-\alpha}\dif z$ be the compensated Poisson random measure.
By L\'evy-It\^o's decomposition, one can write
\begin{align*}
L^\alpha_t=\int_{|z|<1}z{\wt N}(t,\dif z)+\int_{|z|\geq1}z N(t,\dif z).
\end{align*}
Thus SDE \eqref{sde:main} can be written as
\begin{align}
\label{sde:main1}
\dif X_t=b(X_t)\dif t+\int_{|z|<1}\sigma(X_{t-})z{\wt N}(\dif t,\dif z)+\int_{|z|\geq1}\sigma(X_{t-})z N(\dif t,\dif z).
\end{align}
By It\^o's formula (see \cite[Theorem 4.4.7]{App09}) and \eqref{Lya}, we have 
\begin{align}
V_p(X_t)&=V_p(x)+\int^t_0\Big[\langle \nabla V_p(X_s),b(X_s)\rangle+\cL_\sigma V_p(X_s)\Big]\dif s+M_t\label{ito1}\\
&\leq V_p(x)+\int^t_0\Big[-\kappa_0 V_p(X_s)^{1+\frac rp}+\kappa_1\Big]\dif s+M_t,\label{ito2}
\end{align}
where $M_t$ is a local c\'adl\'ag martingale. Let $\tau_n$ be a sequence of stopping times localizing $M_t$, i.e., $M_{t\wedge\tau_n}$ is a martingale and $\tau_n\uparrow \infty$ as $n\to\infty$.
Then we have
$$
\mE V_p(X_{t\wedge\tau_n})+\kappa_0\mE\int^{t\wedge\tau_n}_0 V_p(X_s)^{1+\frac rp}\dif s\leq V_p(x)+\kappa_1 t.
$$
In particular, letting $n\to\infty$ and by Fatou's lemma, we obtain that for all $t\geq 0$,
$$
\mE V_p(X_{t})+\kappa_0\mE\left(\int^{t}_0 V_p(X_s)^{1+\frac rp}\dif s\right)\leq V_p(x)+\kappa_1 t.
$$
Moreover, starting from \eqref{ito2}, by stochastic Gronwall's inequality (\cite[Lemma 3.7]{XZ20}), we also have
$$
\left[\mathbb{E}\left(\sup_{s\in[0,t]}V_p(X_s(x))^{\frac12}\right)\right]^2\lesssim_C V_p(x)+t.
$$
On the other hand, starting from \eqref{ito1}, as above, we have
$$
\mE V_p(X_t)=V_p(x)+\int^t_0\mE\Big[\langle \nabla V_p(X_s),b(X_s)\rangle+\cL_\sigma V_p(X_s)\Big]\dif s.
$$
For $r\geq 0$, by \eqref{Lya} and Jensen's inequality, we have
\begin{align*}
\frac{\dif}{\dif t} \mathbb{E}V_p(X_t)\leq-\kappa_0\mathbb{E}\Big[V_p(X_t)^{1+\frac rp}\Big]+\kappa_1
\leq-\kappa_0\Big[\mathbb{E}V_p(X_t)\Big]^{1+\frac rp}+\kappa_1.
\end{align*}
Let $f(t):=\mathbb{E}V_p(X_t)$.  The above inequality then reads
$$
 f(t)'\leq -\kappa_0 f(t)^{1+\frac rp}+\kappa_1.
$$
If $r=0$, then
\begin{align}\label{LL2}
(\e^{\kappa_0 t}f(t))'\leq \kappa_1\e^{\kappa_0 t}\Rightarrow f(t)\leq \e^{\kappa_0(s-t)}f(s)+\tfrac{\kappa_1}{\kappa_0}(1-\e^{\kappa_0(s-t)}),\ \ t>s\geq 0.
\end{align}
If $r>0$, then by Young's inequality,
\begin{align*}
(t^{\frac{2p}r} f(t))'&\leq \tfrac{2p}r t^{\frac{2p}r-1} f(t)-\kappa_0t^{\frac{2p}r}f(t)^{1+\frac rp}+\kappa_1t^{\frac{2p}r}\leq C t^{\frac{p}r-1}+\kappa_1t^{\frac{2p}r}.
\end{align*}
Integrating both sides from $0$ to $t$, we obtain
$$
t^{\frac{2p}r} f(t)\leq \tfrac{rC}p t^{\frac{p}r}+\tfrac{r\kappa_1}{r+2p} t^{\frac{2p}r+1}\Rightarrow f(t)\leq \tfrac{rC}p t^{-\frac{p}r}+\tfrac{r\kappa_1}{r+2p} t,
$$
which together with \eqref{LL2} yields \eqref{e:x}.
\end{proof}

Let $\phi:\mathbb{R}^d\to\mathbb{R}_+$ be a smooth density function with support in the unit ball. Define
$$
\phi_\eps(x):=\eps^{-d}\phi(\eps^{-1}x),\ \ x\in\mR^d,\ \eps\in(0,1),
$$
and
$$
b_\eps(x):=b*\phi_\eps(x), \quad \sigma_\eps(x):=\sigma*\phi_\eps(x).
$$
By {\bf (H$_{\rm loc}$)}, it is easy to see that for any $m\in\mN$ and all $\eps\in(0,1)$,
\begin{align}\label{AA3}
|b_\eps(x)-b_\eps(y)|+\|\sigma_\eps(x)-\sigma_\eps(y)\|\leq C_{m+1}|x-y|^\gamma,\ \ \forall x,y\in B_m,
\end{align}
and for some $\eps_1\in(0,1)$ depending on $m$, and for all $\eps\in(0,\eps_1)$,
\begin{align}\label{AA4}
C_m^{-1}|\xi|\leq |\sigma_\eps(x)\xi|\leq C_m|\xi|,\ \ \forall x\in B_m,\ \xi\in\mR^d.
\end{align}
Moreover, note that by \eqref{Hol},
\begin{align*}
|b_\eps(x)-b(x)|\leq\int_{\mR^d}|b(x-y)-b(x)|\rho_\eps(y)\dif y\leq \eps^\gamma\ell_1(x),
\end{align*}
and
\begin{align*}
|\sigma_\eps(x)-\sigma(x)|\leq\int_{\mR^d}|\sigma(x-y)-\sigma(x)|\rho_\eps(y)\dif y\leq \eps^\gamma\ell_2(x).
\end{align*}
Hence, by {\bf (H$^{r,q}_{\rm glo}$)}, for any $\eps\in(0,\eps^{1/\gamma}_0)$,
\begin{align}
\<x,b_\eps(x)\>+\KK|\sigma_\eps(x)|^\alpha|x|^{2-\alpha}
&\leq\<x,b(x)\>+\eps^\gamma|x|\ell_1(x)+\KK(\|\sigma(x)\|+\ell_2(x)\eps^\gamma)^\alpha|x|^{2-\alpha}\no\\
&\leq\<x,b(x)\>+\eps_0|x|\ell_1(x)+\KK(\|\sigma(x)\|+\ell_2(x)\eps_0)^\alpha|x|^{2-\alpha}\no\\
&\leq -c_0|x|^{2+r}+c_1.\label{BV1}
\end{align}
Now, consider the following approximation SDE:
\begin{align}
\label{sde:app}
\dif X^{\eps}_t=b_\eps(X^{\eps}_t)\dif t+\sigma_\eps(X^{\eps}_{t-})\dif L^{\alpha}_t,\quad X^{\eps}_0=x.
\end{align}
Since the coefficients are smooth, but, may be  not bounded, by \eqref{BV1} and \cite[Theorem 1.1]{CZZ21}, there is a unique strong solution $X^\eps_t(x)$ to the above SDE. 
Moreover, for any $f\in C_b(\mR^d)$, we have
\begin{align}\label{DD9}
\lim_{\eps\to 0}\mE f(X^\eps_t(x))=\mE f(X_t(x)),
\end{align}
which is proven in Appendix. Notice that if $b$ and $\sigma$ are locally Lipschitz, then it is not necessary to mollify $b$ and $\sigma$.
It suffices to consider the truncated $b$ and $\sigma$ as done below.

Let $\chi$ be a cutoff function with 
$$
\chi(x)=1,\ \ |x|\leq \tfrac{1}{2},\ \ \chi(x)=0,\ \ |x|>1.
$$
For $m\in\mN$, set
$$
\chi_m(x):=\chi(x/m),\ \ b^\eps_m(x):=b_\eps(x)\chi_m(x), \quad \sigma^\eps_m(x):=\sigma_\eps(x\chi_m(x)).
$$
By \eqref{AA3} and \eqref{AA4}, it is easy to see that $b^\eps_m$ and $\sigma^\eps_m$ satisfy the following global assumptions:
\begin{align}\label{AA0}
|b^\eps_m(x)-b^\eps_m(y)|+\|\sigma^\eps_m(x)-\sigma^\eps_m(y)\|\leq C_m|x-y|^\gamma,\ \ \forall x,y\in\mR^d,
\end{align}
and
\begin{align}\label{AA1}
C_m^{-1}|\xi|\leq |\sigma^\eps_m(x)\xi|\leq C_m|\xi|,\ \ \forall x,\xi\in\mR^d,
\end{align}
where the constant does not depend on $\eps$. 
Let $\theta_t(x)$ solve the following ODE:
\begin{align}\label{ode:mp}
\dot \theta_t(x)=(b^\eps_m*\phi_{t^{1/\alpha}})(\theta_t(x)),\quad \theta_0(x)=x.
\end{align}
We also consider the following SDE with  cutoff coefficients:
\begin{align}
\label{sde:app0}
\dif X^{\eps,m}_t=b^\eps_m(X^{\eps,m}_t)\dif t+\sigma^\eps_m(X^{\eps,m}_{t-})\dif L^{\alpha}_t,\quad X^{\eps,m}_0=x.
\end{align}
Let $X^{\eps,m}_t(x)$ be the unique strong solution of the above SDE. By \cite[Theorem 1.1]{MZ20}, $X^{\eps,m}_t(x)$ admits a density $p^\eps_m(t,x,y)$ with the following estimates:
For any $T>0$ and $m\in\mN$, 
there is a constant $C_0\geq 1$ depending on $m$, but {\it independent of} $\eps\in(0,\eps_1)$, such that for all $t\in(0,T]$ and $x,y\in\mathbb{R}^d$,
\begin{align}
\label{est:asy}
p^\eps_m(t,x,y)\asymp_{C_0}t(t^{\frac{1}{\alpha}}+|x-\theta_t(y)|)^{-d-\alpha},
\end{align}
and
\begin{align}
\label{est:gra}
|\nabla \log p^\eps_m(t,x,\cdot)|(y)\leq C_0 t^{-\frac{1}{\alpha}}.
\end{align}
For any $m\in\mN$ and $x\in B_m$, define the exit time of $X^{\eps,m}_t(x)$ from $B_m$ by 
$$
\tau^{\eps,x}_{B_m}:=\inf\big\{t>0:X^{\eps}_t(x)\notin B_m\big\}.
$$
By the uniqueness of strong solution, we have 
\begin{align}
\label{equ:loc}
X^\eps_{t}(x)=X^{\eps,m}_{t}(x),\quad t<\tau^{\eps,x}_{B_m}.
\end{align}

\bl[\bf Strong Feller property] \label{33}
For any $t>0$ and bounded measurable function $f$, the function $x\mapsto P_tf(x)$ is bounded continuous.
\el
\begin{proof}
By a standard monotonic argument, it suffices to show that for any $t>0$,
$$
\lim_{x\to y}\sup_{f\in C_b(\mR^d), \|f\|_\infty\leq 1}|\mathbb{E}(f(X_t(x))-f(X_t(y)))|=0.
$$
Since for any $f\in C_b(\mR^d)$, 
$$
\lim_{\eps\to 0}\mathbb{E}(f(X^\eps_t(x))-f(X^\eps_t(y)))\stackrel{\eqref{DD9}}{=}\mathbb{E}(f(X_t(x))-f(X_t(y))),
$$
furthermore, we only need to prove that
\begin{align}\label{cb}
\lim_{x\to y}\sup_{\eps\in(0,1)}\sup_{f\in C_b(\mR^d), \|f\|_\infty\leq 1}|\mathbb{E}(f(X^\eps_t(x))-f(X^\eps_t(y)))|=0.
\end{align}
Given $x,y\in\mR^d$, let $m>|x|\vee |y|$. For given $f\in C_b(\mR^d)$ with $\|f\|_\infty\leq 1$, we have
\begin{align*}
|\mathbb{E}(f(X^\eps_t(x))-f(X^\eps_t(y)))|
&\leq\big|\mathbb{E}\big[(f(X^\eps_t(x))-f(X^\eps_t(y)))\mathbf{1}_{t<\tau^{\eps,x}_{B_m}\wedge \tau^{\eps,y}_{B_m}}\big]\big|
+2 \mathbb{P}\big(t\geq\tau^{\eps,x}_{B_m}\wedge \tau^{\eps,y}_{B_m}\big)\\
&\!\!\!\!\!\stackrel{\eqref{equ:loc}}{=}\big|\mathbb{E}\big[(f(X^{\eps,m}_t(x))-f(X^{\eps,m}_t(y)))\mathbf{1}_{t<\tau^{\eps,x}_{B_m}\wedge \tau^{\eps,y}_{B_m}}\big]\big|
+2 \mathbb{P}\big(t\geq\tau^{\eps,x}_{B_m}\wedge \tau^{\eps,y}_{B_m}\big)\\
&\leq\big|\mathbb{E}\big[(f(X^{\eps,m}_t(x))-f(X^{\eps,m}_t(y))\big]\big|+2 \mathbb{P}\big(t\geq\tau^{\eps,x}_{B_m}\wedge \tau^{\eps,y}_{B_m}\big).
\end{align*}
By \eqref{est:gra}, we have
\begin{align*}
|\mathbb{E}(f(X^{\eps,m}_t(x))-f(X^{\eps,m}_t(y)))|
&=\left|\int_{\mR^d}f(z) (p^\eps_m(t,x,z)-p^\eps_m(t,x,z))\dif z\right|\\
&\leq \|f\|_{\infty} |x-y|\int_{\mathbb{R}^d}\int^1_0|\nabla p^\eps_m(t,x+r(y-x),z)|\dif z\dif r\\
&\leq C_m t^{-\frac{1}{\alpha}} |x-y|\int^1_0\int_{\mathbb{R}^d} p^\eps_m(t,x+r(y-x),z)\dif z\dif r\\
&= C_m t^{-\frac{1}{\alpha}} |x-y|.
\end{align*}
Moreover, by Chebyshev's inequality and \eqref{e:x1}, 
\begin{align}
\mathbb{P}(t\geq\tau^{\eps,x}_{B_m})\leq \mathbb{P}\left(\sup_{s\leq t}|X^{\eps}_s(x)|\geq m\right)
&\leq \mathbb{E}\left(\sup_{s\leq t}|X^{\eps}_s(x)|^{p/2}\right)/ m^{p/2}\no\\
&\leq C(V_p(x)+t)^{1/2}/m^{p/2}.\label{BV2}
\end{align}
Combining the above calculations, we obtain that for any $f\in C_b(\mR^d)$ with $\|f\|_\infty\leq 1$,
$$
|\mathbb{E}(f(X^\eps_t(x))-f(X^\eps_t(y)))|\leq C_m t^{-\frac{1}{\alpha}} |x-y|+C_0(V_p(x)+V_p(y)+t)^{1/2}/m^{p/2},
$$
where $C_0$ does not depend on $m$.
Thus we obtain \eqref{cb} by first letting $x\to y$ and then $m\to\infty$.
\end{proof}

\bl[{\bf Irreducibility}] \label{34} For any $x_0, y_0\in\mR^d$ and $r, t>0$, we have
\begin{align}
\inf_{x\in B_r(x_0)}\mathbb{P}(X_t(x)\in B_r(y_0))>0.
\end{align}
\el
\begin{proof}
Since $X^\eps_t(x)$ weakly converges to $X_t(x)$ as $\eps\downarrow 0$, for any open set $A\subset\mR^d$, we have
\begin{align}\label{irr}
\mathbb{P}(X_t(x)\in \overline{ A})\geq\liminf_{\eps\to 0}\mathbb{P}(X^\eps_t(x)\in A).
\end{align}
Below we fix $\eps\in(0,1)$ small enough, and $x,y_0\in\mR^d$, $r>0$.
Let $m> 2$ be big enough such that $\{x\}\cup B_r(y_0)\subset B_{m-2}$. Then we have
\begin{align}
\mathbb{P}(X^\eps_t(x)\in B_r(y_0))&\geq \mathbb{P}\(X^{\eps}_t(x)\in B_r(y_0); t<\tau^{\eps,x}_{B_m}\)
\stackrel{\eqref{equ:loc}}{=}  \mathbb{P}\(X^{\eps,m}_{t}(x)\in B_r(y_0); t<\tau^{\eps,x}_{B_m}\).\label{est:lb:ab}
\end{align}
In order to use Theorem \ref{41}, we choose
$$
\rho(t,r)=t(t^{1/\alpha}+r)^{-d-\alpha}.
$$
Clearly, by \eqref{est:asy}, estimate \eqref{DD2} is satisfied for $\Gamma_t(x,y)=|x-\theta_t(y)|$. It remains to find $t_0$ and $\delta_0$ small enough as well as functions $\ell_0(t)$ and $\ell_1(t)$ 
so that the conditions in Theorem \ref{41} are satisfied for  domains $D_0=B_{m-2}$ and $D=B_{m}$.
Note that by \eqref{ode:mp} and the definition of $b^\eps_m$,  for all $t\in[0,1]$ and $y\in\mR^d$,
\begin{align*}
|\theta_t(y)-y|\leq\int^t_0|(b^\eps_m*\phi_{s^{1/\alpha}})(\theta_s(y))|\dif s\leq\sup_{|x|\leq m+1}|b(x)|\cdot t=:C_1 t.
\end{align*}
Let
\begin{align}\label{CC1}
\delta_0\leq \tfrac12\wedge{\rm dist}(y_0,D_0),\ \ t_0\leq \tfrac{1}{2C_1}\wedge\tfrac{\delta_0}{C_1}.
\end{align}
For $x'\notin D$ and $y\in B_{\delta_0}(y_0)\subset D_0$, we have for $t\in(0,t_0]$,
\begin{align*}
\Gamma_t(x',y)=|x'-\theta_t(y)|&\geq |x'-y_0|-|y-y_0|-|\theta_t(y)-y|\geq 2-\delta_0-C_1 t\geq 1.
\end{align*}
On the other hand, for $x,y\in B_{\delta_0}(y_0)$, we have
\begin{align}\label{DD8}
\Gamma_t(x,y)=|x-\theta_t(y)|\leq |x-y|+|\theta_t(y)-y|\leq 2\delta_0+C_1 t\leq 3\delta_0.
\end{align}
Now, for the $C_0$ in \eqref{est:asy},  one can choose $\delta_0, t_0>0$ small enough such that \eqref{CC1} holds and
$$
t\mapsto t/(t^{1/\alpha}+1)^{d+\alpha}\mbox{ is increasing on $(0,t_0]$,}
$$
and
\begin{align}\label{DD6}
\frac{C_0^{-1}t}{(t^{1/\alpha}+3\delta_0)^{d+\alpha}}>\frac{2C_0 t}{(t^{1/\alpha}+1)^{d+\alpha}},\ \ t\in(0,t_0].
\end{align}
In particular, \eqref{AA2} is satisfied. Thus by Theorem \ref{41}, we conclude 
$$
\mathbb{P}\(X^{\eps,m}_{t}(x)\in B_r(y_0); t<\tau^{\eps,x}_{B_m}\)\geq c_0,
$$
where $c_0$ is independent of $\eps$. This, together with \eqref{irr} and \eqref{est:lb:ab}, yields 
$$
\mathbb{P}(X_t(x)\in \overline{B_r(y_0)})\geq c_0.
$$
The proof is thus complete by the strong Feller property that $x\mapsto\mathbb{P}(X_t(x)\in B_r(y_0))$ is continuous.
\end{proof}

Now, we are in a position to give

\begin{proof}[Proof of Theorem \ref{Main}]
First of all, by \eqref{e:x1} and the standard Krylov-Bogoliubov method (see \cite[Theorem 11.7]{Da92}), there is an invariant probability 
measure $\mu$ associated with $(P_t)_{t\geq 0}$.
The uniqueness follows by the strong Feller property and irreducibility. The exponential convergence (i) and (ii) follow by \eqref{e:x}, Theorem \ref{Th11} and Lemmas \ref{33} and \ref{34}.
\end{proof}

\section{Application to heavy-tailed sampling}

In this section we introduce an application of Theorem \ref{Main} to the heavy-tailed sampling. 
Let $\mu(\dif x)=\e^{-U(x)}\dif x/\int_{\mR^d}\e^{-U(x)}\dif x$, where $U:\mR^d\to \mR$ is a continuous function.
Suppose that there are $\beta, C_0>0$ such that for all $|x|\geq 1$,
\begin{align}\label{UU0}
U(x)\geq (d+\beta)\log |x|-C_0.
\end{align}
The above assumption means that $\e^{-U(x)}$ has a polynomial decay rate as $|x|\to\infty$, which is usually called 
the heavy-tailed distribution compared with the exponential or light tailed distribution.
Below we want to find an  ergodic SDE so that the law of  the solution $X_t$ exponentially converges to $\mu$ in some sense 
as $t\to\infty$. Thus, one can sample $\mu$ theoritically from $X_t$ when $t$ is  large.

Fix $\alpha\in(0,2)$. To construct an ergodic SDE driven by $\alpha$-stable process, we introduce a vector field $B: \mR^d\to\mR^d$ by
$$
B(x):=c_{d,\alpha}\int_{\mR^d}y[\varrho_\alpha(x+y)-\varrho_\alpha(x-y)]|y|^{-d-\alpha}\dif y,
$$
where $c_{d,\alpha}=2^\alpha\Gamma(\frac{d+\alpha}{2})/\Gamma(\frac d2)\Gamma(\frac{2-\alpha}2)$ and 
$$
\varrho_\alpha(x):=(1+|x|^2)^{-\frac{d+\alpha}{2}}.
$$
The vector field $B$ enjoys the following important property.
\bl\label{Le41}
\begin{enumerate}[(i)]
\item $B\in C^\infty_0$ satisfies that for some $\kappa_0=\kappa_0(d,\alpha)>0$,
\begin{align}\label{HC2}
|B(x)|\leq \kappa_0|x|^{1-d-\alpha},\ \ |\nabla B(x)|\leq \kappa_0|x|^{-d-\alpha}.
\end{align}
\item There is a constant $\kappa_1>0$ only depending on $d,\alpha$ such that for all $|x|>1$,
\begin{align}\label{HC3}
\<x, B(x)\>\leq -\kappa_1|x|^{2-d-\alpha}.
\end{align}
\item $\div B(x)=\Delta^{\frac\alpha 2}\varrho_\alpha(x)$.
\end{enumerate}
\el

\begin{proof}
(i) Since $\varrho_\alpha\in C^\infty_0(\mR^d)$ and for any $j\in\mN$,
$$
\|\nabla^j\varrho_\alpha/\varrho_\alpha\|_\infty<\infty,
$$
we clearly have $B\in C^\infty_0$. To show the bounds \eqref{HC2}, let $\chi:[0,\infty)\to[0,1]$ be a cutoff function with
$$
\chi(r)=1,\ \ r\in[0,\tfrac14],\ \ \chi(r)=0,\ \ r\in[0,\tfrac12].
$$
Define
$$
\varphi_{|x|}(y):= y|y|^{-d-\alpha}\chi\big(\tfrac{|y|}{|x|}\big),\ \ \wt\varphi_{|x|}(y):=y|y|^{-d-\alpha}\big(1-\chi\big(\tfrac{|y|}{|x|}\big)\big).
$$
For $j=0,1$, by definition we have
\begin{align*}
\nabla^j B(x)&=c_{d,\alpha}\int_{\mR^d}y\Big(\nabla^j\varrho_\alpha(x+y)-\nabla^j\varrho_\alpha(x-y)\Big)|y|^{-d-\alpha}\dif y=c_{d,\alpha}(I_1+I_2),
\end{align*}
where
\begin{align*}
I_1&:= \int_{\mR^d}\varphi_{|x|}(y)\Big(\nabla^j\varrho_\alpha(x+y)-\nabla^j\varrho_\alpha(x-y)\Big)\dif y,\\
I_2&:= \int_{\mR^d}\wt\varphi_{|x|}(y)\Big(\nabla^j\varrho_\alpha(x+y)-\nabla^j\varrho_\alpha(x-y)\Big)\dif y.
\end{align*}
For $I_1$,  we have
\begin{align*}
|I_1|&\leq \int_{\mR^d}\varphi_{|x|}(y)|y|\left(\int^1_{-1}|\nabla^{j+1}\varrho_\alpha(x+sy)|\dif s\right)\dif y\\
&\lesssim \int_{|y|\leq\frac{|x|}2}|y|^{2-d-\alpha}\left(\int^1_{-1}(1+|x+sy|)^{-(d+\alpha+1+j)}\dif s\right)\dif y\\
&\lesssim (1+|x|)^{-(d+\alpha+1+j)}\int_{|y|\leq\frac{|x|}2}|y|^{2-d-\alpha}\dif y\\
&\lesssim (1+|x|)^{-(d+\alpha+1+j)}|x|^{2-\alpha}\lesssim |x|^{1-j-d-2\alpha}.
\end{align*}
For $I_2$, by the change of variable, we have
\begin{align*}
I_2&=\int_{\mR^d}\Big(\wt\varphi_{|x|}(x-y)-\wt\varphi_{|x|}(x+y)\Big)(\nabla^j\varrho_\alpha)(y)\dif y\\
&=|x|^d\int_{\mR^d}\Big(\wt\varphi_{|x|}(|x|(\bar x-y))-\wt\varphi_{|x|}(|x|(\bar x+y))\Big)(\nabla^j\varrho_\alpha)(|x|y)\dif y\\
&=|x|^{1-\alpha}\int_{\mR^d}\Big(\wt\varphi_1(\bar x-y)-\wt\varphi_1(\bar x+y)\Big)(\nabla^j\varrho_\alpha)(|x|y)\dif y,
\end{align*}
where $\bar x=x/|x|$ and in the last step we have used
$$
\wt\varphi_{|x|}(|x|y)=|x|^{1-d-\alpha}y|y|^{-d-\alpha}\big(1-\chi\big(|y|\big)\big)=|x|^{1-d-\alpha}\wt\varphi_1(y).
$$
For $j=0$,  since $|\wt\varphi_1(y)|\leq|y|^{1-d-\alpha}\1_{|y|>1/4}\leq 4^{d+\alpha-1}$, it is easy to see that
\begin{align*}
|I_2|&\leq 2\cdot 4^{d+\alpha-1} |x|^{1-\alpha}\int_{\mR^d}|(1+|x||y|)^{-(d+\alpha)}\dif y\lesssim |x|^{1-\alpha-d}.
\end{align*}
For $j=1$, since $\|\div_y\wt\varphi_1\|_\infty<\infty$, by the integration by parts, we also have
\begin{align*}
|I_2|&=|x|^{-\alpha}\left|\int_{\mR^d}\Big(\wt\varphi_1(\bar x-y)-\wt\varphi_1(\bar x+y)\Big)\nabla_y\varrho_\alpha(|x|y)\dif y\right|\\
&=|x|^{-\alpha}\left|\int_{\mR^d}\Big(\div_y\wt\varphi_1(\bar x-y)-\div_y\wt\varphi_1(\bar x+y)\Big)\varrho_\alpha(|x|y)\dif y\right|\\
&\leq 2\|\div_y\wt\varphi_1\|_\infty |x|^{-\alpha}\int_{\mR^d}|(1+|x||y|)^{-(d+\alpha)}\dif y\lesssim |x|^{-\alpha-d}.
\end{align*}
Combining the above estimates we obtain \eqref{HC2}.

(ii) By definition and the symmetry, we have
\begin{align*}
-\<x,B(x)\>&= c_{d,\alpha}\int_{\mR^d}\<x,y\>[\varrho_\alpha(x-y)-\varrho_\alpha(x+y)]|y|^{-d-\alpha}\dif y\\
&= 2c_{d,\alpha}\int_{\<x,y\>\geq 0}\<x,y\>[\varrho_\alpha(x-y)-\varrho_\alpha(x+y)]|y|^{-d-\alpha}\dif y.
\end{align*}
For $|x|\geq 1$ and $\<x,y\>\geq 0$, by the mean-valued formula, we have
\begin{align*}
&\varrho_\alpha(x-y)-\varrho_\alpha(x+y)
=(1+|x-y|^2)^{-\frac{d+\alpha}2}-(1+|x+y|^2)^{-\frac{d+\alpha}2}\\
&\quad=2(d+\alpha)\<x,y\>\int^1_0\big(1+s(|x-y|^2-|x+y|^2)+|x+y|^2\big)^{-\frac{d+\alpha}2-1}\dif s\\
&\quad=2(d+\alpha)\<x,y\>\int^1_0\big(1-4s\<x,y\>+|x+y|^2\big)^{-\frac{d+\alpha}2-1}\dif s\\
&\quad=2(d+\alpha)\<x,y\>\int^1_0\big(1+4s\<x,y\>+|x-y|^2\big)^{-\frac{d+\alpha}2-1}\dif s \\
&\quad\geq 2(d+\alpha)\<x,y\>\int^{|x|^{-2}}_0\big(1+4s\<x,y\>+|x-y|^2\big)^{-\frac{d+\alpha}2-1}\dif s \\
&\quad\geq 2(d+\alpha)\<x,y\>\big(1+4|x|^{-2}\<x,y\>+|x-y|^2\big)^{-\frac{d+\alpha}2-1}|x|^{-2}. 
\end{align*}
Hence,
\begin{align*}
-\<x,B(x)\>&\geq\frac{4(d+\alpha)c_{d,\alpha}}{ |x|^2}\int_{\<x,y\>\geq 0}\<x,y\>^2\big(1 +4|x|^{-2}\<x,y\>+|x-y|^2\big)^{-\frac{d+\alpha}2-1}|y|^{-d-\alpha}\dif y\\
&\geq\frac{4(d+\alpha)c_{d,\alpha}}{|x|^2}\int_{\<x,y\>\geq 0,|x-y|\leq 1}\<x,y\>^2\big(1 +4|x|^{-2}\<x,y\>+|x-y|^2\big)^{-\frac{d+\alpha}2-1}|y|^{-d-\alpha}\dif y.
\end{align*}
Since for $|x|\geq 1$,
$$
|x-y|\leq 1\Rightarrow \tfrac{|y|^2}{2}\leq\<x,y\>\leq |x|^2+|x|,
$$
we further have
\begin{align*}
-\<x,B(x)\>&\geq\frac{4(d+\alpha)c_{d,\alpha}}{|x|^2}\int_{|x-y|\leq 1}\<x,y\>^2\big(2+4(1+|x|^{-1})\big)^{-\frac{d+\alpha}2-1}|y|^{-d-\alpha}\dif y\\
&\geq\frac{4(d+\alpha)c_{d,\alpha}}{|x|^2 10^{\frac{d+\alpha}2+1}}\int_{|x-y|\leq 1}\<x,y\>^2|y|^{-d-\alpha}\dif y\geq \kappa_1|x|^{2-d-\alpha},
\end{align*}
where $\kappa_1>1$ only depends on $d,\alpha$.  Thus we obtain \eqref{HC3}.

(iii) By definition and the integration by parts, we have
\begin{align*}
\div B(x)&= c_{d,\alpha}\int_{\mR^d}\<y,\nabla_x[\varrho_\alpha(x+y)-\varrho_\alpha(x-y)]\>|y|^{-d-\alpha}\dif y\\
&= c_{d,\alpha}\int_{\mR^d}\<y,\nabla_y[\varrho_\alpha(x+y)+\varrho_\alpha(x-y)-2\varrho_\alpha(x)]\>|y|^{-d-\alpha}\dif y\\
&=- c_{d,\alpha}\int_{\mR^d}\div(y|y|^{-d-\alpha})[\varrho_\alpha(x+y)+\varrho_\alpha(x-y)-2\varrho_\alpha(x)]\dif y.
\end{align*} 
Since for $y\not=0$,
$$
\div(y|y|^{-d-\alpha})=d|y|^{-d-\alpha}+\<y,\nabla|y|^{-d-\alpha}\>=-\alpha|y|^{-d-\alpha},
$$
we have by \eqref{CL11} and $c_{d,\alpha}=2^\alpha\Gamma(\frac{d+\alpha}{2})/\Gamma(\frac d2)\Gamma(\frac{2-\alpha}2)$,
$$
\div B(x)=\alpha c_{d,\alpha}\int_{\mR^d}\frac{[\varrho_\alpha(x+y)+\varrho_\alpha(x-y)-2\varrho_\alpha(x)]}{|y|^{d+\alpha}}\dif y=\Delta^{\frac\alpha2}\varrho_\alpha(x).
$$
The proof is complete.
\end{proof}

\begin{figure}[h]
\begin{center}
\includegraphics[width=150mm]{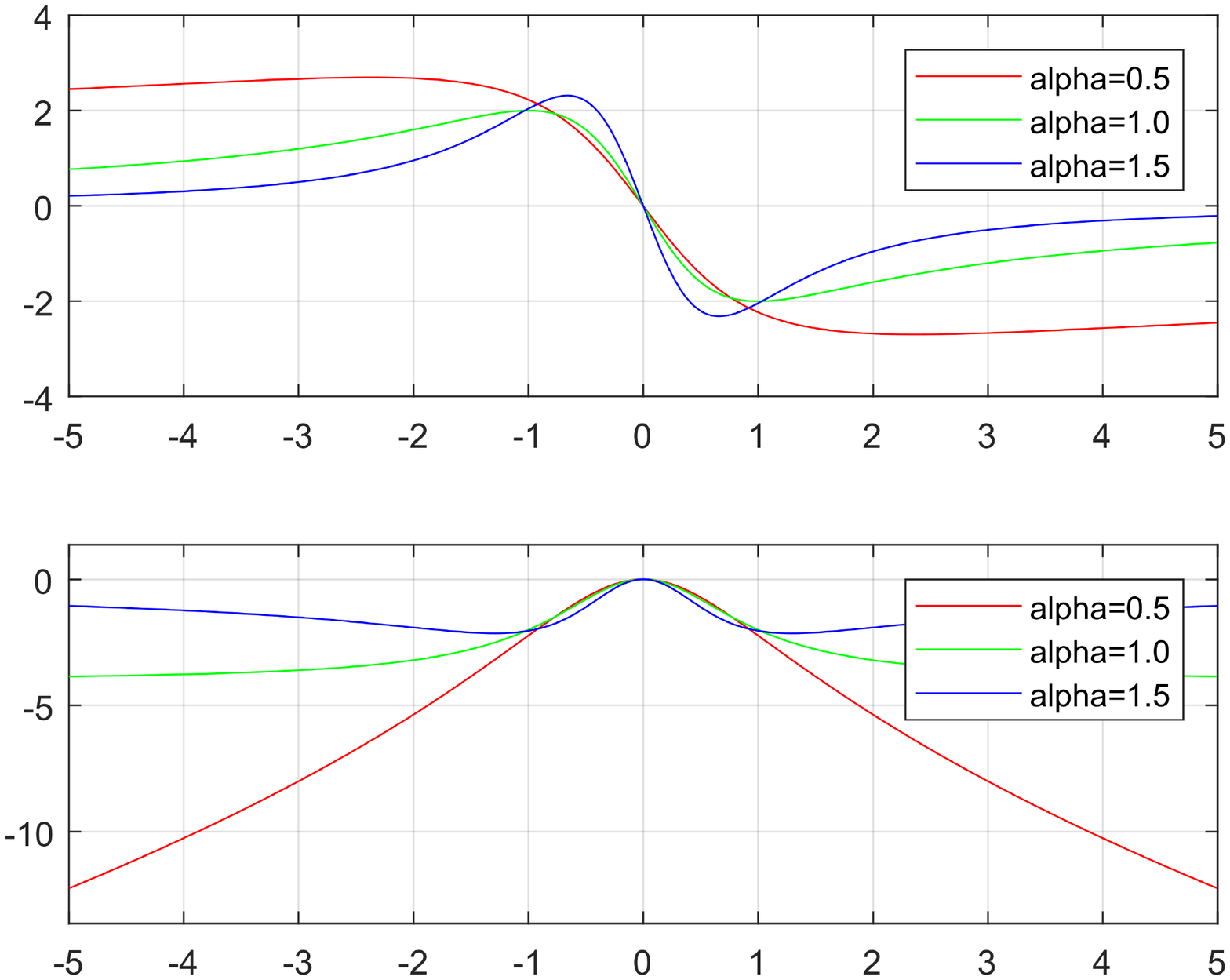}
\end{center}
\caption{Pictures of $B(x)$ and $\<x, B(x)\>$ for $d=1$.}
\label{FIGURE}
\end{figure}

Now we take $\sigma$ and $b$ in SDE \eqref{sde:main} as follows:
\begin{align}\label{Coe}
\sigma(x):=(\varrho_\alpha(x)\e^{U(x)})^{\frac1\alpha}\mI,\ \ b(x):=B(x)\e^{U(x)}.
\end{align}
We have the following main result of this section.
\bt
Suppose that $U$ is locally Lipschitz continuous and satisfies \eqref{UU0}. Then with the above choices of $\sigma$ and $b$,
$\mu(\dif x)=\e^{-U(x)}\dif x/\int_{\mR^d}\e^{-U(x)}\dif x$ 
is the unique invariant probability measure of SDE \eqref{sde:main} and the conclusions in Theorem \ref{Main} hold with $r=\beta-\alpha$.
\et
\begin{proof}
Let $\sigma$ and $b$ be defined by \eqref{Coe}. By \eqref{CL1} and the change of variable, we  can write
\begin{align*}
\mathcal{L}_\sigma f(x)&=\frac{\alpha 2^\alpha\Gamma(\frac{d+\alpha}{2})}{\Gamma(\frac d2)\Gamma(\frac{2-\alpha}2)}\int_{\mathbb{R}^d}[f(x+\sigma(x)z)+f(x-\sigma(x)z)-2f(x)]\dif z/|z|^{d+\alpha}
=\varrho_\alpha(x)\e^{U(x)}\Delta^{\frac\alpha 2}f(x).
\end{align*}
Thus,
\begin{align}\label{Gen1}
\sL f(x)=\mathcal{L}_\sigma f(x)+b\cdot\nabla f(x)=\varrho_\alpha(x)\e^{U(x)}\Delta^{\frac\alpha 2}f(x)+B(x)\e^{U(x)}\cdot\nabla f(x),
\end{align}
and by (i)  of Lemma \ref{Le41},
$$
(\sL^* \e^{-U})(x)=\Delta^{\frac\alpha 2}\varrho_\alpha(x)-\div B(x)=0,
$$
which implies that $\mu(\dif x)=\e^{-U(x)}\dif x/\int_{\mR^d}\e^{-U(x)}\dif x$ is an invariant probability measure of SDE \eqref{sde:main}. 

To check {\bf (H$_{\rm loc}$)} and {\bf (H$^{r,\KK}_{\rm glo}$)}, by Remark \ref{Re8}, it suffices to verify \eqref{BB1} for $\eps_0=0$.
By \eqref{HC3} and \eqref{UU0}, for any $|x|\geq (2q/\kappa_1)^{1/\alpha}\vee 1$, we have
\begin{align*}
\<x,b(x)\>+q\|\sigma(x)\|^\alpha|x|^{2-\alpha}
&=\big[\<x,B(x)\>+\KK\varrho_\alpha(x)|x|^{2-\alpha}\big]\e^{U(x)})\\
&\leq \big[-\kappa_1|x|^{2-d-\alpha}+\KK |x|^{2-d-2\alpha}\big]\e^{U(x)}\\
&\leq -\tfrac{\kappa_1}{2}|x|^{2-d-\alpha}\e^{U(x)}\leq-\tfrac{\kappa_1}{2\e^{C_0}}|x|^{2+\beta-\alpha}.
\end{align*}
So, Theorem \ref{Main} is applicable. The proof is complete.
\end{proof}

\br\rm
Suppose that $U(x)=-\ln\varrho_\alpha(x)$. Note that by \eqref{HC2},
$$
\|\nabla(B/\varrho_\alpha)\|_\infty<\infty.
$$
Let $X_t$ be the unique solution of the following SDE:
\begin{align}\label{SDE8}
\dif X_t=(B/\rho_\alpha)(X_t)\dif t+\dif L^\alpha_t,\ \ X_0=x.
\end{align}
Then $\mu_0(\dif x)=\varrho_\alpha(x)\dif x/\int\varrho_\alpha(x)\dif x$ is the unique invariant probability measure of $X_t$.
It should be noticed that the distributional density $p_\alpha(x)$ of $L^\alpha_1$ is comparable with $\varrho_\alpha(x)$  (cf. \cite{XZ20}), i.e., there is a constant 
$\kappa=\kappa(d,\alpha)\geq 1$ such that for all $x\in\mR^d$,
$$
\kappa^{-1}\varrho_\alpha(x)\leq p_\alpha(x)\leq \kappa\varrho_\alpha(x).
$$
Let $\sL_0$ be the infinitesimal generator of SDE \eqref{SDE8}, i.e.,
$$
\sL_0=\Delta^{\frac{\alpha}{2}}+B/\varrho_\alpha\cdot\nabla.
$$
Then by \eqref{Gen1}, one sees that
$$
\sL=(\varrho_\alpha\e^U)\sL_0.
$$ 
In particular, the solution of SDE \eqref{sde:main} with coefficients \eqref{Coe} is just a time change of SDE \eqref{SDE8} (see \cite[Section 1.15.2]{BGL14}).
\er
\br\rm
The locally Lipschitz assumption on $U$ can be relaxed as locally $\gamma$-H\"older continuity with $\gamma\in((1-\alpha)^+,1)$. 
If it is so, we need to check \eqref{BB1} for some $\eps_0>0$.
\er
\section{Appendix}

In this appendix we sketch the proof  of weak convergence \eqref{DD9}. First of all, by \eqref{BV1} and \eqref{e:x1}, there is a constant $C>0$ such that for all $\eps\in(0,1)$ and
any $x\in\mR^d$ and $T>0$,
\begin{align}\label{e:x11}
\left[\mathbb{E}\left(\sup_{s\in[0,T]}V_p(X^\eps_s(x))^{\frac12}\right)\right]^2\lesssim_C V_p(x)+T.
\end{align}
For $\theta\in(0,T)$, let $\eta,\eta'$ be two stopping times with $0\leq\eta\leq\eta'+\theta\leq T$.
For any $m\in\mN$, we have
\begin{align*}
\mP\left(|X^\eps_\eta-X^\eps_{\eta'}|\geq \delta\right)
&\leq\mP\left(|X^\eps_\eta-X^\eps_{\eta'}|\geq \delta; T<\tau^{\eps,x}_{B_m}\right)+\mP\left(T\geq \tau^{\eps,x}_{B_m}\right)\\
&=\mP\left(|X^{\eps,m}_\eta-X^{\eps,m}_{\eta'}|\geq \delta; T<\tau^{\eps,x}_{B_m}\right)+\mP\left(T\geq \tau^{\eps,x}_{B_m}\right)\\
&\leq\mP\left(|X^{\eps,m}_\eta-X^{\eps,m}_{\eta'}|\geq \delta\right)+2\mP\left(T\geq \tau^{\eps,x}_{B_m}\right).
\end{align*}
By SDE \eqref{sde:app0} and \eqref{AA0}, \eqref{AA1}, it is by now standard to derive that for fixed $m\in\mN$,
\begin{align*}
\lim_{\theta\downarrow 0}\sup_{\eps\in(0,1)}\sup_{0\leq\eta\leq\eta'+\theta\leq T}\mP\left(|X^{\eps,m}_\eta-X^{\eps,m}_{\eta'}|\geq \delta\right)=0,
\end{align*}
which together with \eqref{BV2} yields that for any $T,\delta>0$,
$$
\lim_{\theta\downarrow 0}\sup_{\eps\in(0,1)}\sup_{0\leq\eta\leq\eta'+\theta\leq T}\mP\left(|X^{\eps}_\eta-X^{\eps}_{\eta'}|\geq \delta\right)=0.
$$
Thus, by Aldous' criterion (see \cite[p.356, Theorem 4.5]{Jac03}), the law $\mQ^\eps$ of $(X^\eps_\cdot, L^\alpha_\cdot)$,  $\eps\in(0,1)$ in $\mD\times\mD$ is tight. 
Let $\mQ$ be any accumulation point of $(\mQ^\eps)_{\eps\in(0,1)}$. Without loss of generality, we assume that for some subsequence $\eps_k\to 0$,
$(\mQ^k)_{k\in\mN}:=(\mQ^{\eps_k})_{k\in\mN}$
weakly converges to $\mQ$ as $k\to\infty$. By Skorokhod's representation theorem, 
there is a probability space $(\wt\Omega,\wt\sF,\wt\mP)$ and $\mD\times\mD$-valued processes $(\wt X^k,\wt L^k)$ and $(\wt X,\wt L)$ such that
$$
(\wt X^k,\wt L^k)\to (\wt X,\wt L)\  \mbox{ in $\mD\times\mD$},\ \ \wt\mP-a.s.,
$$
and
$$
\wt\mP\circ(\wt X^k,\wt L^k)^{-1}=\mQ^k,\ \ \wt\mP\circ(\wt X,\wt L)^{-1}=\mQ.
$$
Moreover, $\wt L^k$ and $\wt L$ are still $\alpha$-stable L\'evy processes, and
$$
\wt X_t^k=x+\int^t_0b_{\eps_k}(\wt X^k_s)\dif s+\int^t_0\sigma_{\eps_k}(\wt X^k_s)\dif \wt L^k_s.
$$
By \cite[Theorem 6.22, p.383]{Jac03} and taking limits, one sees that 
$$
\wt X_t=x+\int^t_0b(\wt X_s)\dif s+\int^t_0\sigma(\wt X_s)\dif \wt L_s,
$$
and $(\wt X,\wt L)$ is a weak solution of SDE \eqref{sde:main}. Finally, by the weak uniqueness
of \cite{CZZ21} we obtain the weak convergence \eqref{DD9}.

\medskip

{\bf Acknowledgement:} The authors would like to thank   Jian Wang for useful conversation about the heavy-tailed sampling.


\begin{thebibliography}{10}

\bibitem{App09} Applebaum D.: {\it L\'{e}vy processes and stochastic calculus.} Cambridge Studies in Advanced Mathematics. Vol. 116, Cambridge University Press, Cambridge, 2009.        

\bibitem{BGL14} Bakry D, Gentil I. and Ledoux M.: 
{\it Analysis and geometry of {M}arkov diffusion operators.} Grundlehren der mathematischen Wissenschaften. Vol. 348, Springer, Cham, 2014.

\bibitem{BEL18} Bubeck S., Eldan R. and Lehec J.: Sampling from a log-concave distribution with projected
              Langevin Monte Carlo.  {\it Discrete Comput. Geom.} Vol. 59, No. 4,  pp. 757-783 (2018).


\bibitem{C07}Carmona P.: Existence and uniqueness of an invariant measure for a chain of oscillators in contact with two heat
baths. {\it  Stochastic Process. Appl.} Vol. 117, No. 8, pp. 1076-1092 (2007).

              
\bibitem{C05} Chen M.: 
{\it Eigenvalues, inequalities, and ergodic theory.} Probability and its Applications (New York). Springer-Verlag London, Ltd., London, 2005.
              
\bibitem{Cer01} Cerrai S.: {\it Second order PDE's in finite and infinite dimension.} A probabilistic approach. Vol. 1762, Springer-Verlag, Berlin, 2001.        

\bibitem{CZZ21} Chen Z. Q., Zhang X. and Zhao G.: Supercritical SDEs driven by multiplicative stable-like L\'evy processes. 
{\it Trans. Amer. Math. Soc,} Volume 374, Number 11, November 2021, Pages 7621-7655.

\bibitem{Cheng20} Cheng X., Chatterji N. S., Abbasi-Yadkori Y., Bartlett P. L. and Jordan M. I.: Sharp convergence rates for Langevin dynamics in the nonconvex setting. arXiv: 1805.01648 (2020).

\bibitem{Chu95} Chung K. and Zhao Z.: 
{\it From Brownian motion to Schr\"{o}dinger's equation.} Grundlehren der mathematischen Wissenschaften. Vol. 312, Springer-Verlag, Berlin, 1995.

\bibitem{DK19} Dalalyan A. S. and Karagulyan A.: User-friendly guarantees for the Langevin Monte Carlo with inaccurate gradient. {\it Stochastic Process. Appl.} Vol. 129, No. 12,  pp. 5278-5311 (2019).

\bibitem{DT12} Dalalyan A. S. and Tsybakov A. B.: Sparse regression learning by aggregation and Langevin Monte-Carlo. {\it J. Comput. System Sci.} Vol. 78, No. 5,  pp. 1423-1443 (2012).


\bibitem{Da92} Da Prato G. and Zabczyk J.: 
{\it Stochastic equations in infinite dimensions.} Encyclopedia of Mathematics and its Applications. Vol. 44, Cambridge University Press, Cambridge, 1992.

              
\bibitem{GC11} Girolami M. and Calderhead B.: Riemann manifold Langevin and Hamiltonian Monte Carlo methods. {\it J. R. Stat. Soc. Ser. B Stat. Methodol.} Vol. 73, No. 2,  pp. 123-214 (2011).


\bibitem{GM06} Goldys B. and Maslowski B.: {\it Exponential ergodicity for stochastic reaction-diffusion equations.} Stochastic partial differential equations and applications-{VII}. Vol. 245, Chapman \& Hall/CRC, Boca Raton, FL, 2006.  

\bibitem{GSZ21}G\"urb\"uzbalaban M., Simsekli U. and Zhu L.: The heavy-tail phenomenon in SGD. {\it International Conference on Machine Learning.} PMLR.  pp. 3964-3975 (2021).
                                                    
\bibitem{HMW21} Huang L., Majka M. and Wang J.: Approximation of heavy-tailed distributions via stable-driven SDEs. {\it Bernoulli.} Vol. 27, No. 3,  pp. 2040-2068 (2021).   
              
\bibitem{Jac03} Jacod J. and Shiryaev A. N.: 
{\it Limit theorems for stochastic processes.} Grundlehren der mathematischen Wissenschaften. Vol. 288, Springer-Verlag, Berlin, Second edition, 2003.

\bibitem{K09} Kulik A. M.: Exponential ergodicity of the solutions to SDE's with a jump noise. {\it Stochastic Process. Appl.} Vol. 119, No. 2,  pp. 602-632 (2009).                    
  
 \bibitem{LW20} Liang M. and Wang J.: Gradient estimates and ergodicity for SDEs driven by multiplicative L\'evy noises via coupling. {\it Stochastic Process. Appl.} Vol. 130, No. 5,  pp. 3053-3094 (2020).
             
             
\bibitem{LMW21} Liang M., Majka M. and Wang J.: Exponential ergodicity for SDEs and McKean-Vlasov processes with L\'evy noise. {\it Ann. Inst. Henri Poincar\'{e} Probab. Stat.} Vol. 57, No. 3,  pp. 1665-1701 (2021).                    

\bibitem{LW19} Luo D. and Wang J.: Refined basic couplings and {W}asserstein-type distances for {SDE}s with {L}\'{e}vy noises. {\it Stochastic Process. Appl.} Vol. 129, No. 9,  pp. 3129-3173 (2019).


\bibitem{M17} Majka  M. B.: Coupling and exponential ergodicity for stochastic differential equations driven by {L}\'{e}vy processes. {\it Stochastic Process. Appl.} Vol. 127, No. 12,  pp. 4083-4125 (2017).


\bibitem{M07} Masuda H.: Ergodicity and exponential {$\beta$}-mixing bounds for multidimensional diffusions with jumps. {\it Stochastic Process. Appl.} Vol. 117, No. 1,  pp. 35-56 (2007).


\bibitem{MSH02}Mattingly J.C., Stuart A.M. and Higman D.J.: Ergodicity for SDEs and approximations: locally Lipschitz vector fields and degenerate noise.
{\it Stoch. Proc. Appl.} Vol. 101, pp. 185-232 (2002) .
              
\bibitem{MZ20} Menozzi S. and Zhang X.: Heat kernel of supercritical SDEs with unbounded drifts. 
to appear in {\it Journal de l'\'Ecole polytechnique-Math\'ematiques}, arXiv: 2012.14775 (2020).  

\bibitem{MT09} Meyn S. and Tweedie R. L.: {\it Markov chains and stochastic stability.} Cambridge University Press, Cambridge, Second edition, 2009.   
              

\bibitem{Rev99} Revuz D. and Yor M.: {\it Continuous martingales and Brownian motion.} Vol. 293, Springer-Verlag, Berlin, 1999.    
          
\bibitem{Rob96} Roberts G. O. and Tweedie R. L.: Exponential convergence of {L}angevin distributions and their
              discrete approximations. {\it Bernoulli.} Vol. 2, No. 4,  pp. 341-363 (1996).              

\bibitem{S17}Simsekli U.: Fractional langevin monte carlo: Exploring l{\'e}vy driven stochastic differential equations for markov chain monte carlo. {\it International Conference on Machine Learning.} PMLR.  pp. 3200-3209 (2017).

\bibitem{SZTG20}Simsekli U., Zhu L, Teh Y.W. and G\"urb\"uzbalaban M.: Fractional underdamped langevin dynamics: Retargeting sgd with momentum under heavy-tailed gradient noise. {\it International Conference on Machine Learning.} PMLR.  pp. 8970-8980 (2020).

\bibitem{S79}Stein E.M.: {\it Singular integrals and differentiability properties of functions.} Princeton Mathematical Series, No. 30, Princeton University Press, Princeton, N.J., 1970.  

\bibitem{NSR19}Nguyen T., Simsekli U. and Richard, D.: Non-asymptotic analysis of Fractional Langevin Monte Carlo for non-convex optimization. {\it International Conference on Machine Learning.} PMLR.  pp. 4810-4819 (2019).
     
\bibitem{W06} Wang F.: {\it Functional inequalities Markov semigroups and spectral theory.} Elsevier, 2006.
     
\bibitem{Wang16} Wang J.: $L^p$-Wasserstein distance for stochastic differential equations driven by L\'{e}vy processes. {\it Bernoulli.} Vol. 22, No. 3,  pp. 1598-1616 (2016).              
              
\bibitem{XZ20} Xie L. and Zhang X.: Ergodicity of stochastic differential equations with jumps and singular coefficients. {\it Ann. Inst. Henri Poincar\'{e} Probab. Stat.} Vol. 56, No. 1,  pp. 175-229 (2020).
              
\bibitem{YZ18}Ye N. and Zhu Z.: Stochastic Fractional Hamiltonian Monte Carlo. {\it Proceedings of the 27th International Joint Conference on Artificial Intelligence.}  pp. 3019-3025 (2018).


\end{thebibliography}
\end{document}